\documentclass[11pt]{article}
\usepackage[T1]{fontenc}
\usepackage[margin=1in]{geometry}
\usepackage{amsmath,amssymb,amsthm,mathtools}
\usepackage[authoryear]{natbib}
\usepackage{mathrsfs,bm}
\usepackage{booktabs,tabularx,array}
\usepackage{url}
\usepackage[colorlinks,citecolor=blue,urlcolor=blue,linkcolor=blue]{hyperref}
\numberwithin{equation}{section}
\theoremstyle{plain}
\newtheorem{theorem}{Theorem}[section]
\newtheorem{proposition}[theorem]{Proposition}
\newtheorem{lemma}[theorem]{Lemma}
\newtheorem{corollary}[theorem]{Corollary}

\theoremstyle{definition}
\newtheorem{definition}[theorem]{Definition}
\newtheorem{example}[theorem]{Example}
\newtheorem{remark}[theorem]{Remark}
\newtheorem{fact}[theorem]{Fact}
\newcommand{\N}{\mathbb N}
\newcommand{\Nzero}{\mathbb N_0}
\newcommand{\E}{\mathbb E}
\newcommand{\F}{\mathcal F}
\newcommand{\M}{\mathcal M}
\newcommand{\Iuniv}{\mathcal I}
\newcommand{\C}{\mathcal C}
\newcommand{\D}{\mathcal D}
\newcommand{\Rfam}{\mathfrak R}
\newcommand{\Hh}{\mathsf H}
\newcommand{\Hhsup}{\overline{\mathsf H}}
\newcommand{\FDP}{\operatorname{FDP}}
\newcommand{\calN}{\mathcal N}
\newcommand{\loss}{L}
\newcommand{\runloss}{\overline L}
\newcommand{\can}{\mathrm{can}}
\newcommand{\one}{\mathbf 1}

\newcommand{\ceBH}{\mathrm{ceBH}}

\newcommand{\proofidea}{\par\medskip\noindent\textit{Proof idea.}\ }
\newcommand{\justification}{\par\medskip\noindent\textit{Justification.}\ }
\title{Dynamic $e$-closure for online hypotheses with any-time-valid evidence: closure principles and projective mergers}
\author{Rianne de Heide\\Department of Applied Mathematics, University of Twente\\ and Centrum Wiskunde \& Informatica, Amsterdam\\\texttt{r.deheide@utwente.nl}}
\date{}
\begin{document}
\maketitle
\begin{abstract}
Many modern testing problems are sequential along two axes: new hypotheses may arrive over time, while evidence for hypotheses already under consideration continues to evolve and may be inspected at arbitrary stopping times. We develop dynamic $e$-closure for this setting. At a global stopping time the active true-null intersection is random. Future-extension coherence allows its certificate to be compared with that of a fixed terminal intersection, yielding simultaneous stopped-FDR control. If the certificates are also time-monotone, the resulting closure controls simultaneous SupFDR and is setwise persistent. Conversely, every procedure satisfying either criterion is contained in a dynamic closure generated by canonical normalized-loss processes. For pointwise mergers, fixed-dimensional admissibility is equivalent to ordinary arbitrary-dependence $e$-merging. Coherence across horizons then forces a single globally summable weight sequence and exact neutrality under padding by the $e$-value one; on a countably infinite hypothesis universe, this rules out nontrivial symmetric mergers in the admissible pointwise class. The theory extends from FDP to bounded losses that are monotone in the possible true-null configuration and local in the reported action. We also give a coherence counterexample, persistent constructions, and a globally valid shared-control model.
\end{abstract}
\tableofcontents
\newpage
\section{Introduction}\label{sec:intro}

Many testing problems are open-ended in two ways. In a platform trial, new treatment arms may be introduced while earlier comparisons continue to recruit patients. An experimentation platform may start new product experiments while previous ones are still running, and a scientific screening programme may keep adding hypotheses instead of fixing the complete family in advance. The hypothesis family therefore grows while the evidence for hypotheses already present continues to change.

Closed testing is a natural starting point because it separates local evidence for intersection hypotheses from the final multiple-testing procedure. Classical closed testing starts from valid intersection tests \citep{MarcusPeritzGabriel1976}; fixed-time $e$-closure starts from intersection $e$-values and thereby accommodates expectation-based losses such as the FDR \citep{XuEtAl2026Closure}. We develop the corresponding theory when hypotheses may arrive over time and evidence for existing hypotheses may continue to update, with validity preserved under optional stopping. Neither a final hypothesis list nor a final sample size needs to be fixed in advance.

The two temporal axes interact. Reapplying fixed-time closure after every evidence update does not justify a decision made at a global stopping time: the active true-null intersection at that time is random, whereas the $e$-process guarantee is stated for each fixed intersection. A second issue arises across hypothesis horizons. A merger that is valid at every fixed dimension can change when a new, still uninformative coordinate is added.

We resolve the stopping-time problem with future-extension coherence and a localization argument that compares the random active true-null certificate with one fixed terminal-intersection process. This yields simultaneous stopped FDR and, with time-monotone certificates, simultaneous SupFDR and persistence. We then study pointwise merger rules that are required to be compatible
under neutral padding for every finite inclusion of coordinate sets.
Combining this all-subset compatibility condition with the
fixed-dimensional characterization of admissible $e$-mergers of
\citet{Wang2025Merging} shows that all finite-dimensional weights must
be restrictions of one globally summable sequence.

\subsection{Motivating example: a continuously operating platform trial}\label{sec:platform-motivation}

Platform trials provide a simple example in which the two forms of sequentiality occur together. Under a common master protocol, several interventions may be compared with a shared control, while new treatment arms are introduced and recruitment to existing arms continues \citep{LeeEtAl2021Platform,RobertsonWasonRamdas2023}. For each arm $i$, let $H_i$ denote the null hypothesis that the treatment has no beneficial effect relative to control. When a new arm opens, a new hypothesis enters the family. At the same time, outcomes continue to arrive for arms that were already open, so the evidence for these hypotheses keeps evolving after later hypotheses have been introduced.

Shared control also makes the arm-wise statistics dependent, and a monitoring committee may base stopping or reporting decisions on the complete trial history. An $e$-process that is valid only in an arm-specific filtration need not remain valid for such a global stopping rule; validity is needed in the filtration used to stop \citep{WangDandapanthulaRamdas2025,ChoeRamdas2026}. Section~\ref{sec:adjuster} describes the filtration-lifting result of \citet{ChoeRamdas2026}. Section~\ref{sec:shared-control} instead verifies global validity directly in a simple Gaussian shared-control model. That model leaves aside random arm additions, delayed outcomes, time trends, nonconcurrent controls, and regulatory decision rules.

\subsection{Terminology: online, any-time, simultaneous, persistent, and coherent}\label{sec:terminology}

The word \emph{online} is used differently in neighboring literatures. In online learning and bandit theory it often refers to observations arriving round by round, whereas in multiple testing it usually refers to hypotheses arriving in a stream \citep{FosterStine2008,FischerBofillBrannath2024}. We use the latter convention. Thus \emph{online} refers to the arrival of new hypotheses, while \emph{any-time} refers to continued sampling or evidence updates for hypotheses already present. We use \emph{any-time valid} only for methods, guarantees, or evidence objects whose validity is preserved under optional stopping \citep{RamdasWangBook}.

The word \emph{simultaneous} means that several candidate rejection sets may be certified at the same time and that the guarantee holds for every set in that family, as in post-hoc closed testing and $e$-closure \citep{GoemanSolari2011,XuEtAl2026Closure}. The maximum over the candidate family below expresses this guarantee. We use \emph{persistent} for the separate property that a rejection set certified earlier remains certified after further evidence updates.

The word \emph{coherence} has several statistical uses. In classical simultaneous testing it refers to compatibility with logical implication between hypotheses \citep{Gabriel1969}. Online closure uses the temporal notion of \emph{predictability}: adding only future hypotheses to an intersection should not undo a rejection already made \citep{FischerBofillBrannath2024}. Definition~\ref{def:future-coherence} imposes the corresponding one-sided inequality on numerical intersection certificates.

A single global update index records all observed events. The active-set process records which hypotheses are currently available, whereas the filtration records the full evidence history. A stopping time defined adapted to that filtration is called a \emph{global stopping time}. This term emphasizes that the decision to stop may use information from every active stream, including shared controls and other common data.

\subsection{Three meanings of closure}\label{sec:closure-meanings}

The word \emph{closure} is used in three related senses. Classical closed testing augments the elementary hypotheses by their intersections and derives global decisions from local intersection tests \citep{MarcusPeritzGabriel1976}. Following \citet{XuEtAl2026Closure}, the forward $e$-closure map instead takes intersection evidence as input and returns a family of certified decisions. This map is not an order-theoretic closure operator because its input and output are different objects. An order-theoretic \emph{hull} appears in Section~\ref{sec:closure-operator}: it starts from a candidate-family procedure and adds every decision whose configuration-wise loss is no larger than the loss already authorized by that procedure. Appendix~\ref{sup:closure-proof} proves the closure-operator properties.

The forward and converse results give a necessary-and-sufficient representation \emph{up to containment}, as in fixed-time $e$-closure \citep{XuEtAl2026Closure}. Every dynamic closure generated from valid local evidence satisfies the stated error criterion, and every procedure satisfying that criterion is contained in a closure generated by canonical local evidence. The representing closure may be larger than the original procedure. This representational statement is separate from Proposition~\ref{prop:coherence-necessary}, which concerns the assumptions needed for a universal forward theorem.

\subsection{What must be formulated for the joint setting}\label{sec:ingredients}

The construction uses three existing ingredients: fixed-time $e$-closure \citep{XuEtAl2026Closure}, the predictability condition from online closure \citep{FischerBofillBrannath2024}, and $e$-process validity in the filtration used for stopping \citep{WangDandapanthulaRamdas2025}. The task is to combine these ingredients in a way that is compatible with
both an expanding hypothesis family and continuing evidence.

For this we introduce the \emph{dynamic intersection $e$-process collection},  indexed by both the intersection hypothesis and global evidence time, together with a condition that we call \emph{future-extension coherence}. This condition says that adding hypotheses that are not yet active cannot decrease the current intersection certificate. This gives \emph{horizon invariance}: computing the closure
from the currently active family gives the same result as first extending to any finite future horizon
and then restricting back to the current hypotheses. Stopping requires one further argument,
because the active true-null intersection at a stopping time is random. The \emph{random intersection localization lemma} controls this random object through one fixed terminal intersection, to which
the usual e-process validity applies. A maximal version of the same argument yields SupFDR
control.

To construct intersection processes from coordinate processes, we study \emph{pointwise $e$-process mergers}, which use only the current coordinate values. When such
merger rules are specified for every finite intersection, they must also determine what happens when
a new coordinate enters with the neutral value one. We call exact invariance under this operation \emph{neutral-padding projectivity}. Section~\ref{sec:rigidity} imposes the stronger requirement that
neutral-padding coherence hold for every finite inclusion $B\subseteq S$.
Within the admissible arbitrary-dependence pointwise class, this
all-subset condition forces one globally summable weight sequence.

\subsection{Main contributions}\label{sec:contributions}

The main results are the following.
\begin{enumerate}
\item \textbf{Dynamic $e$-closure principles.} We prove horizon invariance and random-intersection localization, and then derive forward dynamic $e$-closure theorems for simultaneous stopped FDR and, after adding time-monotone certificates, simultaneous SupFDR and setwise persistence. Conversely, every procedure satisfying either guarantee is contained in a closure generated by canonical normalized-loss $e$-processes. Hence dynamic $e$-closure is necessary and sufficient up to containment, in the same sense as the fixed-time representation of \citet{XuEtAl2026Closure}. A separate two-hypothesis construction shows that coherence cannot be removed from a distribution-free forward theorem, even when every fixed-intersection process is a test martingale. We also show that, once the candidate family itself is persistent, fixed-time simultaneous FDR, stopped FDR, and SupFDR coincide. As in the fixed-time theory of \citet{XuEtAl2026Closure}, the arguments extend beyond FDP to suitable bounded expected losses.
\item \textbf{Mergers across horizons.} We show that a Borel map merges arbitrary $e$-processes pointwise in a common global filtration if and only if it is an ordinary arbitrary-dependence $e$-merging function. Combining this equivalence with \citeauthor{Wang2025Merging}'s fixed-dimensional characterization needs the \emph{cross-horizon rigidity theorem}: coherence forces one globally summable weight sequence and, within the admissible pointwise class, implies exact \emph{neutral-padding projectivity}. On a countably infinite universe, no nontrivial symmetric merger remains in this class.

\item \textbf{Constructions.} We give a globally weighted projective closure, persistent constructions based on adjusted running maxima, and a globally valid shared-control model. We also characterize pathwise nondecreasing $e$-processes by a terminal expectation condition and give closed eBH and closed BY as batch any-time benchmarks.

\item \textbf{Canonical minimality and hulls.} The canonical current-loss and running-loss processes are the smallest coherent certificates covering a given procedure. Applying the forward closure map to these certificates defines current and running hulls. A fixed point of that hull is called \emph{loss-saturated}: it already contains every rejection set that can be added without increasing its own current or running worst-case loss profile.

\end{enumerate}

Throughout, the active sets are deterministic and finite, and sequential validity is defined relative to one global filtration. The general closure constraints can be exponential in the number of active hypotheses. Random predictable arrivals, history-dependent mergers, persistent calibration, and scalable computation are left for future work.

\subsection{Related work}\label{sec:related}

Classical closure organizes multiplicity adjustment through intersection hypotheses \citep{MarcusPeritzGabriel1976}; later work used this structure for simultaneous true-discovery statements, post-hoc inference, consonance, and admissibility \citep{GoemanSolari2011,GoemanHemerikSolari2021}. Online multiple testing has roots in alpha-investing \citep{FosterStine2008} and later online FDR procedures \citep{JavanmardMontanari2018}. Online closure \citep{FischerBofillBrannath2024} extends closed testing to a growing hypothesis stream through a predictability condition that prevents future hypotheses from invalidating current decisions. \citet{FischerRamdas2024} study online true-discovery guarantees using any-time-valid intersection tests. \citet{FischerXuRamdas2024OnlineBH} develop horizon-aware online versions of BH and eBH; ordinary fixed-dimensional eBH is recovered when the number of hypotheses is known to be fixed.

The eBH procedure of \citet{WangRamdas2022} controls FDR under arbitrary dependence using $e$-values. Fixed-time $e$-closure \citep{XuEtAl2026Closure} replaces binary intersection tests by intersection $e$-values and allows a family of rejection sets to be certified simultaneously. It also gives the fixed-time necessary-and-sufficient representation up to containment used as the starting point for our converse results. \citet{SunWang2026Admissibility} give a fixed-dimensional complete-class result for weighted-mean closed-eBH procedures based on arbitrary $e$-values. 

\citet{XuFischerRamdas2026Online} develop an online SupFDR $e$-closure
principle for the standard online setting in which hypothesis $H_i$
arrives together with a single $e$-value $E_i$. Their intersection
evidence consists of one $e$-value $E^S$ for each finite $S$, and is
required to be increasing when $S$ is extended only by later-arriving
indices. This yields nested closure families and SupFDR control, and
they use suitable increasing $e$-collections to improve e-LOND and
related procedures. They also develop a separate donation framework
based on online compound $e$-values and give efficient algorithms for
the resulting procedures.

Our setting adds a second evidence clock. After a hypothesis has arrived, its evidence may continue to change, so each intersection is
equipped with an entire process $(E_t^S)_{t\geq0}$ rather than one
$e$-value $E^S$. At a global stopping time $\tau$, both the evaluation
time and the active true-null index $\calN_\tau(P)$ are therefore
random. Fixed-intersection $e$-process validity controls evaluation at
the random time; Lemma~\ref{lem:random-intersection} uses
future-extension coherence to remove the random intersection index.

The online closure theorem of \citet{XuFischerRamdas2026Online} is
recovered as a one-axis persistent special case of our construction.
Take $I_t=[t]$ and let $(E^S)$ be their increasing $e$-collection.
Defining
\[
\widehat E_t^S:=E^{S\cap[t]}
\]
(with the empty-set bookkeeping convention) gives a time-monotone
intersection $e$-process collection. Its time monotonicity follows
from increasingness of $(E^S)$, and for $S\subseteq[t]$ one has
$\widehat E_t^S=E^S$, so the resulting closure at time $t$ is their
online $e$-closure. The additional work here is needed when evidence
for an already active intersection continues to evolve between
hypothesis arrivals.

The second temporal axis comes from sequential inference. Test martingales go back to Ville \citep{Ville1939} and are developed as measures of evidence by \citet{ShaferEtAl2011TestMartingales}; we use the $e$-process terminology summarized in \citet{RamdasWangBook}. In multiple testing, \citet{WangDandapanthulaRamdas2025} study validity at global stopping times, \citet{ChoeRamdas2026} show how adjusters can lift $e$-processes to a finer filtration, and \citet{TavyrikovGoemanDeHeide2026} use running-maximum adjustment for persistent testing. Definition~\ref{def:supfdr} distinguishes the SupFDR criterion used here from their FDR-sup criterion.

For $e$-value merging, \citet{VovkWang2021} study symmetric mergers, \citet{Wang2025Merging} characterize all admissible arbitrary-dependence mergers at a fixed dimension, and \citet{Clerico2026Merging} give a shorter geometric proof of the latter result. \citet{VovkWang2024Sequential} study sequential merging of conditionally valid one-step $e$-values. Our inputs are instead several already-formed $e$-processes evolving in parallel. Theorem~\ref{thm:process-merger-equivalence} reduces pointwise merging to Wang's fixed-dimensional problem, and Theorem~\ref{thm:coherent-merger-rigidity} adds compatibility across horizons.

\section{Setup and preliminaries}\label{sec:setup}

\subsection{A global filtration with two temporal axes}

Let $(\Omega,\F,(\F_t)_{t\in\Nzero})$ be a discrete-time filtered measurable space, and let $\M$ be a family of probability measures on $(\Omega,\F)$. The global update index $t$ may record the arrival of a new hypothesis or new observations for hypotheses that are already active.

Let $\Iuniv\subseteq\N$ be an at most countable index set, which may be finite, and let $H_i\subseteq\M$, $i\in\Iuniv$, be null hypotheses. For $P\in\M$, define the true-null index set
\begin{equation}\label{eq:true-nulls}
  \calN(P):=\{i\in\Iuniv:P\in H_i\}.
\end{equation}
For every finite $S\subseteq\Iuniv$, let
\begin{equation}\label{eq:intersection-null}
  H_S:=\bigcap_{i\in S}H_i,
  \qquad H_\varnothing:=\M.
\end{equation}

At time $t$, a deterministic finite set $I_t\subseteq\Iuniv$ records the hypotheses that have arrived. We assume
\begin{equation}\label{eq:active-sets}
  I_0\subseteq I_1\subseteq\cdots,
  \qquad |I_t|<\infty,
  \qquad \bigcup_{t\geq0}I_t=\Iuniv.
\end{equation}
The active true nulls under $P$ are
\begin{equation}\label{eq:active-true}
  \calN_t(P):=\calN(P)\cap I_t.
\end{equation}
If $\Iuniv=\N$, $I_t=[t]$, and each hypothesis contributes one final statistic, we recover the standard online multiple-testing axis. If $\Iuniv=[m]$ and $I_t=[m]$ for all $t$, the family is batch but its evidence may still be updated and inspected at any time. In the general case the two types of updates are interleaved.

\subsection{Candidate rejection families and loss profiles}

Following the formulation of \citet{BenjaminiHochberg1995}, for a finite rejection set $R$ and a possible true-null configuration $S$, define the \emph{False Discovery Proportion}
\begin{equation}\label{eq:fdp}
  \FDP_S(R):=\frac{|S\cap R|}{|R|\vee1}.
\end{equation}
We repeatedly use two elementary properties. If $S\subseteq T$, then
\begin{equation}\label{eq:fdp-monotone}
  \FDP_S(R)\leq\FDP_T(R),
\end{equation}
because the numerator can only increase while the denominator is unchanged. Moreover, if $R\subseteq I_t$, then
\begin{equation}\label{eq:fdp-locality}
  \FDP_S(R)=\FDP_{S\cap I_t}(R).
\end{equation}

At time $t$, a procedure may certify a nonempty random family
\[
  \C_t\subseteq2^{I_t}.
\]
We assume $\{R\in\C_t\}\in\F_t$ for every deterministic $R\subseteq I_t$. Since $I_t$ is finite, all maxima below are measurable. A standard procedure reporting one set $R_t$ is the special case $\C_t=\{R_t\}$.

The family formulation follows the post-hoc simultaneous-inference viewpoint of closed testing and fixed-time $e$-closure \citep{GoemanSolari2011,GoemanHemerikSolari2021,XuEtAl2026Closure}: after observing the data, any set in $\C_t$ may be selected while retaining the stated guarantee. For any $S\subseteq\Iuniv$, define the current loss profile
\begin{equation}\label{eq:current-loss}
  \loss_t^S(\C):=\max_{R\in\C_t}\FDP_{S\cap I_t}(R)
\end{equation}
and the running loss profile
\begin{equation}\label{eq:running-loss}
  \runloss_t^S(\C):=\max_{0\leq u\leq t}\loss_u^S(\C).
\end{equation}
Both are nondecreasing in $S$, whereas only $\runloss_t^S$ is necessarily nondecreasing in $t$.

\subsection{Stopped FDR and SupFDR}

All stopping-time statements below are understood distribution by distribution: for fixed $P$, an ``almost surely finite'' stopping time $\tau$ satisfies $P(\tau<\infty)=1$.

\begin{definition}[Simultaneous stopped FDR]\label{def:stopped-fdr}
The candidate-family process $\C=(\C_t)_{t\geq0}$ controls simultaneous stopped FDR at level $\alpha\in(0,1]$ if, for every $P\in\M$ and every almost surely finite $(\F_t)$-stopping time $\tau$,
\begin{equation}\label{eq:stopped-fdr}
  \E_P\!\left[\loss_\tau^{\calN(P)}(\C)\right]\leq\alpha.
\end{equation}
\end{definition}

For singleton candidate families, this is the stopped-FDR criterion of \citet{WangDandapanthulaRamdas2025}. For a genuine candidate family, ``simultaneous'' means that more than one rejection set may be certified and the guarantee applies to every certified set; the maximum in Equation~\eqref{eq:stopped-fdr} records the worst FDP among them.

\begin{definition}[Simultaneous SupFDR]\label{def:supfdr}
The candidate-family process $\C$ controls simultaneous SupFDR at level $\alpha$ if, for every $P\in\M$,
\begin{equation}\label{eq:supfdr}
  \E_P\!\left[\sup_{t\geq0}\loss_t^{\calN(P)}(\C)\right]\leq\alpha.
\end{equation}
\end{definition}

The name SupFDR follows \citet{XuRamdas2022Dynamic}, who use $\E[\sup_t\FDP(R_t)]$ for a single online rejection sequence. Definition~\ref{def:supfdr} additionally maximizes over the candidate family at each time. This differs from the FDR-sup criterion of \citet{TavyrikovGoemanDeHeide2026}, which applies a multiple-testing procedure to coordinatewise running suprema of the $e$-processes and controls the FDR of the resulting persistent rejection set. SupFDR instead takes the supremum of the FDP path before expectation. Hence SupFDR implies stopped FDR; Example~\ref{ex:stopped-not-sup} shows that the converse fails.

\subsection{Intersection \texorpdfstring{$e$}{e}-process collections}

A nonnegative random variable $X$ is an $e$-variable for a null hypothesis $H$ if $\E_P[X]\leq1$ for every $P\in H$. An $e$-process requires the same expectation bound after every admissible stopping time.

\begin{definition}[$e$-process and test supermartingale]\label{def:eprocess}
Following \citet{RamdasWangBook}, a nonnegative adapted process $E=(E_t)_{t\geq0}$ is an $e$-process for a null hypothesis $H\subseteq\M$, relative to $(\F_t)$, if
\begin{equation}\label{eq:eprocess}
  \E_P[E_\tau]\leq1
\end{equation}
for every $P\in H$ and every almost surely finite $(\F_t)$-stopping time $\tau$. We allow $E_0\leq1$ rather than requiring $E_0=1$, because a nontrivial pathwise nondecreasing $e$-process cannot start at one; see Proposition~\ref{prop:monotone-eprocess-characterization}. A \emph{test supermartingale} for $H$ is a nonnegative adapted process that is a supermartingale under every $P\in H$ and has initial value at most one. The use of nonnegative martingales as sequential tests goes back to Ville \citep{Ville1939}; the terminology and evidence interpretation of test martingales are developed explicitly by \citet{ShaferEtAl2011TestMartingales}.
\end{definition}

Every test supermartingale is an $e$-process by optional stopping, but the converse need not hold. When the distinction matters, we use the descriptive phrase \emph{non-supermartingale $e$-process} for an $e$-process that is not itself a test supermartingale. This does not assert that the process lacks a supermartingale majorant; such domination results are part of the general theory of $e$-processes \citep{RamdasRufLarssonKoolen2020}.

\begin{definition}[Dynamic intersection $e$-process collection]\label{def:intersection-system}
A \emph{dynamic intersection $e$-process collection} is a family
\[
  \mathbf E=(E_t^S)_{t\geq0,\;S\subseteq\Iuniv,\;0<|S|<\infty}
\]
such that $E^S=(E_t^S)_{t\geq0}$ is an $e$-process for the intersection null $H_S$ for every nonempty finite $S$.
\end{definition}

Definition~\ref{def:intersection-system} imposes validity separately for each fixed intersection. The next condition relates the certificates at different hypothesis horizons.

\begin{definition}[Future-extension coherence]\label{def:future-coherence}
A dynamic intersection $e$-process collection is \emph{future-extension coherent} if, pointwise on $\Omega$, for every time $t$ and every finite $S$ with $S\cap I_t\neq\varnothing$,
\begin{equation}\label{eq:future-coherence}
  E_t^{S\cap I_t}\leq E_t^S.
\end{equation}
\end{definition}

Equation~\eqref{eq:future-coherence} says that adding hypotheses that
have not yet arrived cannot decrease the certificate available now.
It is the numerical analogue of the no-reversal condition in online
closure \citep{FischerBofillBrannath2024}; here it is imposed at every
evidence update because evidence for active hypotheses continues to
evolve. When $I_t=[t]$ and evidence is frozen after entry, this reduces
to the ordered increasingness condition used in online $e$-closure
\citep{XuFischerRamdas2026Online}: an intersection is compared only
with extensions by later-arriving hypotheses.

\begin{definition}[Neutral-padding projective consistency]\label{def:projective-consistency}
A dynamic intersection $e$-process collection is \emph{neutral-padding projectively consistent}, or simply \emph{projectively consistent}, if it can be extended to the empty index set by an adapted bookkeeping process $E^\varnothing$ such that, pointwise on $\Omega$, for every finite $S$ and every time $t$,
\begin{equation}\label{eq:projective}
  E_t^S=E_t^{S\cap I_t}.
\end{equation}
The empty-set extension is never used as evidence for a nonempty intersection null.
\end{definition}

Projective consistency is the equality version of coherence: inactive indices have no effect on the current certificate. It is stronger than the one-sided conditions used in online closure. The term \emph{projective} is explained in Section~\ref{sec:weighted-construction}, where the consistency maps are explicit neutral-padding embeddings.

The next condition concerns evolution of the certificate over evidence time.

\begin{definition}[Time monotonicity]\label{def:time-monotone}
A dynamic intersection $e$-process collection is \emph{time-monotone} if, pointwise on $\Omega$,
\begin{equation}\label{eq:time-monotone}
  E_t^S\leq E_u^S\qquad(t\leq u)
\end{equation}
for every finite nonempty $S$.
\end{definition}

Time monotonicity is an additional persistence condition. A certificate may increase as evidence evolves but cannot later decrease. This requires the subnormalized convention above: if $E_0=1$ and $E_t$ is pathwise nondecreasing, then $E_t\geq1$ almost surely, while deterministic-time $e$-validity gives $\E[E_t]\leq1$; hence $E_t=1$ almost surely. Likewise, a pathwise nondecreasing test supermartingale is constant; see Appendix~\ref{sup:betting}. Adjusted running maxima provide one construction of nontrivial time-monotone $e$-processes, but not every such process has that form; Appendix~\ref{sup:nonadjuster} gives a counterexample.

\begin{definition}[Setwise persistence]\label{def:setwise-persistence}
A candidate-family process is \emph{setwise persistent} if, pointwise on $\Omega$ and after identifying every $R\subseteq I_t$ with the same subset of $I_u$ for $u\geq t$,
\[
  \C_t\subseteq\C_u\qquad(t\leq u).
\]
\end{definition}

Setwise persistence is the candidate-family analogue of the nonretraction requirement in carefree multiple testing \citep{TavyrikovGoemanDeHeide2026}: every previously certified set remains certified. It does not imply that independently selected rejection sets are nested. Nested reporting requires a separate selection rule that restricts later choices to supersets of the previously reported set.

\begin{proposition}[Characterization of time-monotone $e$-processes]\label{prop:monotone-eprocess-characterization}
Let $V=(V_t)_{t\geq0}$ be a nonnegative, adapted, pathwise nondecreasing process, and let $H$ be a null hypothesis. The following are equivalent:
\begin{enumerate}
\item[(i)] $V$ is an $H$-$e$-process;
\item[(ii)] $\E_P[V_t]\leq1$ for every deterministic $t$ and every $P\in H$;
\item[(iii)] with $V_\infty:=\sup_{t\geq0}V_t$, one has $\E_P[V_\infty]\leq1$ for every $P\in H$.
\end{enumerate}
Thus, for a nonnegative adapted process that can only increase with time, full stopping-time validity is equivalent to the single requirement that its limiting value $V_\infty=\sup_t V_t$ has expectation at most one under every null distribution.
\end{proposition}

\begin{proof}
The implication (i)$\Rightarrow$(ii) follows by taking a deterministic stopping time. If (ii) holds, then $V_t\uparrow V_\infty$, so monotone convergence gives
\[
  \E_P[V_\infty]=\lim_{t\to\infty}\E_P[V_t]\leq1.
\]
Finally, (iii) implies (i) because $V_\tau\leq V_\infty$ pathwise for every almost surely finite stopping time $\tau$.
\end{proof}

For a pathwise nondecreasing certificate, deterministic-time expectation bounds already imply validity at every stopping time. Equivalently, it is enough to check $\E_P[V_\infty]\leq1$. The construction problem is therefore to make $V_t$ respond strongly to the data while keeping the terminal expectation within this budget of $1$. Adjusted running maxima are one general construction.

\section{Dynamic \texorpdfstring{$e$}{e}-closure principles}\label{sec:dynamic-principles}

\subsection{Horizon invariance and the random-intersection problem}\label{sec:random-intersection}

Future-extension coherence has two consequences needed below. It makes current decisions invariant to the choice of a larger finite future horizon, and it allows the certificate for the random active true-null intersection at a stopping time to be compared with a fixed-intersection $e$-process.

Following the fixed-time loss-constrained $e$-closure of \citet{XuEtAl2026Closure}, given a dynamic intersection collection $\mathbf E$, define the current closure family
\begin{equation}\label{eq:dynamic-closure}
  \Rfam_{\alpha,t}(\mathbf E)
  :=\left\{R\subseteq I_t:
  \FDP_S(R)\leq\alpha E_t^S
  \text{ for every nonempty }S\subseteq I_t\right\}.
\end{equation}
For a finite $J\subseteq\Iuniv$ with $I_t\subseteq J$, also define the restriction of the $J$-horizon closure to current decisions by
\begin{equation}\label{eq:horizon-closure}
  \Rfam_{\alpha,t}^{J}(\mathbf E)
  :=\left\{R\subseteq I_t:
  \FDP_{S\cap I_t}(R)\leq\alpha E_t^S
  \text{ for every }S\subseteq J\text{ with }S\cap I_t\neq\varnothing\right\}.
\end{equation}
Write $\Rfam_\alpha(\mathbf E):=(\Rfam_{\alpha,t}(\mathbf E))_{t\geq0}$ for the resulting candidate-family process.

\begin{proposition}[Horizon invariance]\label{prop:horizon-invariance}
If $\mathbf E$ is future-extension coherent, then for every $t$ and every finite $J\subseteq\Iuniv$ with $I_t\subseteq J$,
\begin{equation}\label{eq:horizon-invariance}
  \Rfam_{\alpha,t}^{J}(\mathbf E)=\Rfam_{\alpha,t}(\mathbf E).
\end{equation}
If the collection is projectively consistent, every future-horizon constraint is exactly the corresponding current constraint after inactive indices are removed.
\end{proposition}

\begin{proof}
The inclusion from left to right follows by taking $S\subseteq I_t$ in Equation~\eqref{eq:horizon-closure}. Conversely, let $R\in\Rfam_{\alpha,t}(\mathbf E)$ and fix $S\subseteq J$ with $B:=S\cap I_t\neq\varnothing$. Current feasibility and future-extension coherence give
\[
  \FDP_{S\cap I_t}(R)=\FDP_B(R)
  \leq\alpha E_t^B\leq\alpha E_t^S.
\]
Hence $R\in\Rfam_{\alpha,t}^{J}(\mathbf E)$. Under projective consistency $E_t^S=E_t^B$, so the two constraints coincide.
\end{proof}

Thus the closure computed from the active family is exactly the restriction of the closure obtained at any larger finite horizon. Without coherence, adding a dormant future hypothesis can tighten a constraint on a decision that is already available.

For a single padded constraint, with $B=S\cap I_t$, the current and future right-hand sides are $\alpha E_t^B$ and $\alpha E_t^S$. Requiring the future constraint never to be tighter is exactly the inequality $E_t^B\leq E_t^S$. A particular violation can be irrelevant if the corresponding constraint never binds, but fixed-intersection validity alone cannot support a general forward theorem, as the next proposition shows.

\begin{proposition}[Necessity of coherence for the forward theorem]\label{prop:coherence-necessary}
There exist two true null hypotheses, deterministic active sets, and a dynamic intersection collection for which every fixed-intersection process is a nonnegative test martingale, but the induced dynamic $e$-closure at level $\alpha=1/2$ has stopped FDR equal to one. Hence fixed-intersection $e$-process validity alone does not imply the forward dynamic $e$-closure theorem.
\end{proposition}

\begin{proof}
Let the statistical model consist of one distribution $P$ on $\Omega=\{a,b\}$ with $P(a)=P(b)=1/2$, and let both null hypotheses equal the full model. Put $I_0=\varnothing$, $I_1=\{1\}$, and $I_2=\{1,2\}$, with $I_t=I_2$ thereafter. Let $\F_0$ be trivial and $\F_t=\sigma(\{a\})$ for $t\geq1$.

For each nonempty $S\subseteq\{1,2\}$, set $E_0^S=1$ and, for $t\geq1$, define
\[
 E_t^{\{1\}}=2\one_{\{a\}},\qquad
 E_t^{\{2\}}=2\one_{\{b\}},\qquad
 E_t^{\{1,2\}}=2\one_{\{b\}}.
\]
Each process has conditional expectation one at the first update and is constant afterwards, so every $E^S$ is a nonnegative test martingale for $H_S$.

Define the global stopping time
\[
 \tau=\begin{cases}1,&\omega=a,\\2,&\omega=b.\end{cases}
\]
On $\{a\}$, at time one the set $R=\{1\}$ satisfies its only nonempty closure constraint because
\[
 \FDP_{\{1\}}(R)=1=\tfrac12 E_1^{\{1\}}.
\]
On $\{b\}$, at time two the set $R=\{2\}$ satisfies all three constraints: the $\{1\}$-constraint has loss zero, while
\[
 \FDP_{\{2\}}(R)=\FDP_{\{1,2\}}(R)=1
 =\tfrac12E_2^{\{2\}}=\tfrac12E_2^{\{1,2\}}.
\]
Thus the dynamic closure contains an all-null rejection set at the stopping time on every sample path. Hence
\[
 \max_{R\in\Rfam_{1/2,\tau}(\mathbf E)}\FDP_{\{1,2\}}(R)=1
 \quad\text{almost surely},
\]
and the stopped FDR equals one.

The failure is precisely a coherence violation. At time one on $\{a\}$,
\[
 E_1^{\{1\}}=2>0=E_1^{\{1,2\}},
\]
so the stopping rule selects favourable evidence from different intersection indices on different paths. Each deterministic-index process is valid, but the random-index process is not controlled.
\end{proof}

\begin{remark}[Scope of the necessity statement]\label{rem:coherence-necessity}
Theorem~\ref{thm:forward-stopped} proves sufficiency of fixed-intersection $e$-process validity plus future-extension coherence. Proposition~\ref{prop:coherence-necessary} shows that the first condition alone is insufficient: a stopping rule can select different valid intersection processes on different sample paths and thereby violate stopped FDR. The proposition does not claim that every noncoherent collection fails; a coherence violation may be irrelevant if its constraint never binds. The statement is just about what can be guaranteed uniformly from fixed-intersection validity alone.
\end{remark}

For fixed $P$, the active true-null set $\calN_t(P)$ changes with $t$. Although $E^S$ is an $e$-process for every deterministic nonempty $S$, fixed-intersection validity does not by itself control the certificate indexed by $\calN_\tau(P)$: at the stopping time, both the time and the intersection index are data-dependent. Optional stopping controls a fixed $e$-process evaluated at a random time; it does not allow the data to choose among different $e$-processes.

Localization removes the random index. Fix a deterministic horizon $T$. Under fixed $P$, the terminal true-null set $\calN_T(P)$ is a deterministic finite set. For any stopping time bounded by $T$,
\[
  \calN_\tau(P)=\calN_T(P)\cap I_\tau.
\]
On the event $\calN_\tau(P)\neq\varnothing$, future-extension coherence therefore gives
\[
  E_\tau^{\calN_\tau(P)}\leq E_\tau^{\calN_T(P)}.
\]
The right-hand side is one fixed $e$-process evaluated at the random time $\tau$, so its expectation is controlled. Proposition~\ref{prop:coherence-necessary} shows why this replacement is needed: without coherence, a stopping rule can switch between valid intersection processes on different sample paths.

\begin{lemma}[Random-intersection localization]\label{lem:random-intersection}
Let $\mathbf E$ be future-extension coherent. For fixed $P\in\M$, define
\[
  Z_t(P):=\one\{\calN_t(P)\neq\varnothing\}E_t^{\calN_t(P)}.
\]
Then for every almost surely finite global stopping time $\tau$,
\begin{equation}\label{eq:random-intersection-stopped}
  \E_P[Z_\tau(P)]\leq1.
\end{equation}
If $\mathbf E$ is also time-monotone, then
\begin{equation}\label{eq:random-intersection-max}
  \E_P\!\left[\sup_{t\geq0}Z_t(P)\right]\leq1.
\end{equation}
\end{lemma}

\proofidea
For a bounded stopping time $\tau\leq T$, the random active true-null set satisfies $\calN_\tau(P)=\calN_T(P)\cap I_\tau$. Future-extension coherence therefore bounds its certificate by the fixed-intersection process $E^{\calN_T(P)}$ at the same stopping time. Ordinary optional stopping applies to that fixed process. The maximal statement uses time monotonicity to dominate every earlier certificate by the single terminal value $E_T^{\calN_T(P)}$. The complete argument, including passage to unbounded almost surely finite stopping times, is in Appendix~\ref{app:proof-localization}.

The lemma is the only step that couples hypothesis arrival with optional stopping. Fixed-intersection $e$-process validity controls the random time, and future-extension coherence removes the random intersection index. Time monotonicity extends the same comparison from one stopping time to the maximum over time.

\subsection{Why stopped FDR is not enough for persistence}\label{sec:separation}

The next example shows why the persistent principle requires stronger control than validity at every stopping time.

\begin{example}[Stopped FDR without SupFDR]\label{ex:stopped-not-sup}
Let $\alpha=1/2$. There are two hypotheses, with $H_1$ true and $H_2$ false, and both are active from time zero. At time zero, before a fair coin $X$ is observed, the candidate family is
\[
  \C_0=\bigl\{\{1,2\}\bigr\}.
\]
At time one it becomes
\[
  \C_1=
  \begin{cases}
    \bigl\{\{1\}\bigr\},&X=1,\\
    \bigl\{\{2\}\bigr\},&X=0,
  \end{cases}
\]
and remains unchanged thereafter. Let $\F_0$ be trivial and $\F_t=\sigma(X)$ for $t\geq1$.

The FDP is $1/2$ at time zero. At time one it is one when $X=1$ and zero when $X=0$, so its expectation is again $1/2$. Because no nontrivial event is observable at time zero, every stopping rule either stops immediately on all paths or reaches time one before making a data-dependent decision. Hence stopped FDR is controlled at level $1/2$.

The pathwise maximum is different:
\[
  \sup_t\FDP(R_t)=
  \begin{cases}
    1,&X=1,\\
    1/2,&X=0,
  \end{cases}
\]
so
\begin{equation}\label{eq:separation}
  \E\!\left[\sup_t\FDP(R_t)\right]
  =\frac12\cdot1+\frac12\cdot\frac12
  =\frac34>\frac12.
\end{equation}
\end{example}

In this example the canonical current-loss certificate for the true intersection equals one at time zero and equals either two or zero at time one. It is valid at every stopping time, but its running maximum has expectation $3/2$. Thus stopped validity of a certificate does not imply validity of its running maximum. The candidate family is not persistent: the set $\{1,2\}$, which is certified at time zero, is no longer certified at time one.When the candidate family is setwise persistent, the loss profile itself is nondecreasing and the distinction between fixed-time FDR, stopped FDR, and SupFDR disappears; see Proposition~\ref{prop:persistent-fdr-equivalence}.

\subsection{The forward validity theorems}\label{sec:forward}

The forward theorem now follows by applying the closure constraint to the active true-null intersection and then invoking Lemma~\ref{lem:random-intersection}.

\begin{theorem}[Dynamic $e$-closure principle for stopped FDR]\label{thm:forward-stopped}
Let $\mathbf E$ be a future-extension coherent dynamic intersection $e$-process collection. Then $\Rfam_\alpha(\mathbf E)$ controls simultaneous stopped FDR at level $\alpha$: for every $P\in\M$ and every almost surely finite global stopping time $\tau$,
\begin{equation}\label{eq:forward-stopped-result}
  \E_P\!\left[
  \max_{R\in\Rfam_{\alpha,\tau}(\mathbf E)}
  \FDP_{\calN(P)}(R)
  \right]\leq\alpha.
\end{equation}
Consequently, any $\F_\tau$-measurable selection from the candidate family controls stopped FDR.
\end{theorem}

\proofidea
At $\tau$, the closure constraint gives
$\max_{R\in\Rfam_{\alpha,\tau}(\mathbf E)}\FDP_{\calN(P)}(R)\leq\alpha Z_\tau(P)$, with both sides zero when no true null is active. Lemma~\ref{lem:random-intersection} gives $\E_P[Z_\tau(P)]\leq1$. The complete proof is in Appendix~\ref{app:proof-forward-stopped}.

\begin{theorem}[Persistent dynamic $e$-closure principle]\label{thm:forward-persistent}
Let $\mathbf E$ be future-extension coherent and time-monotone. Then:
\begin{enumerate}
\item[(a)] the candidate process controls simultaneous SupFDR,
\begin{equation}\label{eq:forward-sup-result}
  \E_P\!\left[
  \sup_{t\geq0}\max_{R\in\Rfam_{\alpha,t}(\mathbf E)}
  \FDP_{\calN(P)}(R)
  \right]\leq\alpha;
\end{equation}
\item[(b)] the candidate families are setwise persistent: $R\in\Rfam_{\alpha,t}(\mathbf E)$ and $u\geq t$ imply $R\in\Rfam_{\alpha,u}(\mathbf E)$.
\end{enumerate}
\end{theorem}

\proofidea
The closure inequality gives a pathwise bound by $\alpha Z_t(P)$ at every time, and the maximal part of Lemma~\ref{lem:random-intersection} controls its supremum. Time monotonicity also preserves every constraint satisfied earlier, which gives setwise persistence. The complete proof is in Appendix~\ref{app:proof-forward-persistent}.

\begin{remark}[Reduction to the existing one-axis settings]\label{rem:one-axis-reductions}
If $I_t=[m]$ is fixed, the growing-horizon condition disappears. At a global stopping time $\tau$, the stopped intersection processes are valid intersection $e$-values, so Theorem~\ref{thm:forward-stopped} reduces to fixed-time $e$-closure applied to those values \citep{XuEtAl2026Closure}.

If $I_t=[t]$ and each intersection contributes one $e$-value that is
fixed once available, the separate evidence clock disappears. In
particular, let $(E^S)$ be an increasing $e$-collection in the sense of
\citet{XuFischerRamdas2026Online} and define
\[
\widehat E_t^S:=E^{S\cap[t]}.
\]
Then $\widehat E^S$ is time-monotone, future-extension coherent, and an
$e$-process, and for every $S\subseteq[t]$ its value at time $t$ is
$E^S$. Hence Theorem~\ref{thm:forward-persistent} recovers their online
SupFDR $e$-closure theorem. With continuing evidence updates,
fixed-intersection validity must also hold at global stopping times,
and Lemma~\ref{lem:random-intersection} handles the resulting random
intersection index.
\end{remark}

\begin{remark}[Persistence of the candidate family versus persistent reporting]
	\label{rem:persistence-selection}
	Setwise persistence means that every rejection set certified at an earlier
	time remains certified later. It does not require a researcher who selects
	one set from the candidate family at each time to make nested selections.
	For example, a set $R_t$ may remain available at time $u>t$, while the
	researcher chooses a different certified set $R_u$ that does not contain
	$R_t$. If previously reported discoveries should never be withdrawn, the
	selection rule must therefore impose this separately, for instance by
	requiring $R_t\subseteq R_u$ whenever $t\leq u$.
\end{remark}

\subsection{Necessary-and-sufficient characterization}\label{sec:ns-characterization}

The converse direction is obtained by normalizing the loss profile of a procedure that already satisfies the target error criterion. This gives the following characterizations, with equivalence understood up to containment as in Section~\ref{sec:closure-meanings}.

\begin{theorem}[Necessary and sufficient dynamic $e$-closure for stopped FDR]\label{thm:characterization-stopped}
For a candidate-family process $\C$, the following are equivalent:
\begin{enumerate}
\item[(i)] $\C$ controls simultaneous stopped FDR at level $\alpha$;
\item[(ii)] there exists a future-extension coherent dynamic intersection $e$-process collection $\mathbf E$ such that
\begin{equation}\label{eq:characterization-stopped}
  \C_t\subseteq\Rfam_{\alpha,t}(\mathbf E)
  \qquad\text{for every }t.
\end{equation}
\end{enumerate}
\end{theorem}

\proofidea
The implication (ii)$\Rightarrow$(i) is Theorem~\ref{thm:forward-stopped}. For (i)$\Rightarrow$(ii), normalize the procedure's current worst-case FDP loss to obtain the canonical intersection processes of Theorem~\ref{thm:converse-stopped}; those processes are coherent $e$-processes and their closure contains the original candidate family. The complete proof is in Appendix~\ref{app:proof-char-stopped}.

\begin{theorem}[Necessary and sufficient dynamic $e$-closure for SupFDR]\label{thm:characterization-sup}
For a candidate-family process $\C$, the following are equivalent:
\begin{enumerate}
\item[(i)] $\C$ controls simultaneous SupFDR at level $\alpha$;
\item[(ii)] there exists a future-extension coherent, time-monotone dynamic intersection $e$-process collection $\mathbf E$ such that
\begin{equation}\label{eq:characterization-sup}
  \C_t\subseteq\Rfam_{\alpha,t}(\mathbf E)
  \qquad\text{for every }t.
\end{equation}
\end{enumerate}
Every collection in (ii) also generates setwise persistent closure families.
\end{theorem}

\proofidea
The implication (ii)$\Rightarrow$(i), together with setwise persistence, follows from Theorem~\ref{thm:forward-persistent}. For (i)$\Rightarrow$(ii), use the canonical running-loss processes of Theorem~\ref{thm:converse-persistent}, which are coherent and time-monotone and whose closure contains the original procedure. The complete proof is in Appendix~\ref{app:proof-char-sup}.

These theorems extend the fixed-time representation of \citet{XuEtAl2026Closure} to the joint online/any-time setting. Every procedure satisfying the corresponding error criterion is contained in a coherent dynamic $e$-closure. 


\subsection{Beyond FDP: bounded monotone local losses}\label{sec:general-losses}

The FDP is not essential to the preceding arguments. The proofs only use two of its properties: the loss cannot decrease when more hypotheses are treated as true nulls, and only true nulls contained in the rejection set can contribute to the loss. Thus, as in the fixed-time setting of \citet{XuEtAl2026Closure}, let $f_S(R)\in[0,1]$ be a bounded loss satisfying
\begin{equation}\label{eq:general-loss-monotone}
	S\subseteq T\quad\Longrightarrow\quad f_S(R)\leq f_T(R),
\end{equation}
and
\begin{equation}\label{eq:general-loss-local}
	f_S(R)=f_{S\cap R}(R).
\end{equation}
Define the current and running loss profiles as before, with $f_S(R)$ in place of $\FDP_S(R)$, and replace the dynamic closure constraint by
\begin{equation}
	f_S(R)\leq \alpha E_t^S.
\end{equation}

\begin{theorem}[Dynamic closure for bounded monotone local losses]\label{thm:general-loss}
	Under Equations~\eqref{eq:general-loss-monotone}--\eqref{eq:general-loss-local}, the stopped-loss and persistent supremum-loss results, the canonical converse representations up to containment, and the pointwise minimality results remain valid with $f$ in place of FDP. The corresponding current and running hulls remain closure operators.
\end{theorem}

Indeed, when $R\subseteq I_t$, locality gives
\begin{equation}
	f_S(R)=f_{S\cap I_t}(R),
\end{equation}
which is the only additional identity needed in the forward proofs. Monotonicity in $S$ gives the comparison needed in the converse and minimality arguments. Appendix~\ref{sup:losses} gives the details. Examples include the familywise-error loss $f_S(R)=\one\{|S\cap R|>0\}$ and weighted FDP losses with nonnegative weights.

\subsection{Setwise persistence removes the distinction between error criteria}

For a setwise persistent candidate family, the current loss profile is itself nondecreasing in time. This makes the fixed-time, stopped, and supremum criteria equivalent.

\begin{proposition}[Persistence collapses the fixed-time, stopped, and supremum criteria]\label{prop:persistent-fdr-equivalence}
Suppose that $\C$ is setwise persistent. Then the following are equivalent at level $\alpha$:
\begin{enumerate}
\item[(i)] for every $P\in\M$ and every deterministic time $t$,
\[
  \E_P\!\left[\loss_t^{\calN(P)}(\C)\right]\leq\alpha;
\]
\item[(ii)] $\C$ controls simultaneous stopped FDR in the sense of Definition~\ref{def:stopped-fdr};
\item[(iii)] $\C$ controls simultaneous SupFDR in the sense of Definition~\ref{def:supfdr}.
\end{enumerate}
\end{proposition}

\begin{proof}
Fix $P\in\M$. Setwise persistence makes $t\mapsto\loss_t^{\calN(P)}(\C)$ pathwise nondecreasing. Indeed, if $s\leq t$ and $R\in\C_s$, then $R\in\C_t$, and because $R\subseteq I_s$,
\[
  \FDP_{\calN(P)\cap I_s}(R)
  =\FDP_{\calN(P)\cap I_t}(R).
\]
Hence (i) implies (iii) by monotone convergence. The implication (iii)$\Rightarrow$(ii) follows from
\[
  \loss_\tau^{\calN(P)}(\C)
  \leq \sup_{t\geq0}\loss_t^{\calN(P)}(\C),
\]
and (ii)$\Rightarrow$(i) follows by taking $\tau=t$ deterministic.
\end{proof}

Thus, for persistent families, fixed-time control already bounds the pathwise maximum. The separation in Example~\ref{ex:stopped-not-sup} is possible only because certified actions may disappear.

\subsection{A practical design recipe}

The preceding results suggest the following practical recipe:
\begin{enumerate}
\item \textbf{Specify the information set.} Choose the global filtration and construct coordinate or intersection processes that are valid at stopping times in that filtration. If only local-filtration processes are available, the lifting route of Section~\ref{sec:adjuster} can be used first.
\item \textbf{Make future hypotheses compatible.} Choose intersection evidence or a merger family whose treatment of dormant future hypotheses satisfies future-extension coherence; projective consistency is a convenient stronger condition.
\item \textbf{Apply closure.} At every update, certify exactly the rejection sets satisfying the deterministic constraints in Equation~\eqref{eq:dynamic-closure}.
\item \textbf{Add persistence only if needed.} For SupFDR and nonretractable certification, use a time-monotone intersection collection. A nested selected rejection sequence is an additional reporting choice and is not required for the candidate-family guarantee.
\end{enumerate}
Caveat: as in ordinary closed testing, the general closure constraints can be exponential in the number of active hypotheses. Efficient algorithms are for future work.

\section{Admissible \texorpdfstring{$e$}{e}-process mergers across growing horizons}\label{sec:mergers}

We now construct intersection processes from coordinate $e$-processes. For each fixed coordinate set, a pointwise merger must preserve $e$-process validity under optional stopping. Across coordinate sets, the merger rules must also be compatible when new hypotheses enter. The first requirement reduces to the fixed-dimensional $e$-merging problem; the second yields the cross-horizon rigidity theorem.

\subsection{Pointwise mergers of arbitrary \texorpdfstring{$e$}{e}-processes}

Fix a positive integer $K$. Following the $e$-merging terminology of \citet{VovkWang2021,Wang2025Merging}, an \emph{arbitrary-dependence $e$-merging function} is an increasing Borel map $F:[0,\infty)^K\to[0,\infty)$ such that $F(X_1,\ldots,X_K)$ is an $e$-variable whenever $X_1,\ldots,X_K$ are arbitrary, possibly dependent, $e$-variables. Throughout this subsection, all process inputs are adapted to the same filtration and are valid for the same null hypothesis. This common-filtration requirement is essential: an $e$-process valid only in a smaller local filtration need not be valid at a stopping time based on a larger filtration \citep{WangDandapanthulaRamdas2025}. Section~\ref{sec:adjuster} discusses the lifting construction of \citet{ChoeRamdas2026} for coordinate processes that are initially available only in local filtrations.

\begin{definition}[Pointwise $e$-process merger]\label{def:process-merger}
An increasing Borel map $F:[0,\infty)^K\to[0,\infty)$ is a \emph{pointwise $e$-process merger} if, for every filtered statistical model, every null hypothesis $H$, and every $K$-tuple of $H$-$e$-processes $M^1,\ldots,M^K$ in the same filtration, the process
\begin{equation}\label{eq:pointwise-process-merger}
  t\longmapsto F(M_t^1,\ldots,M_t^K)
\end{equation}
is an $H$-$e$-process.
\end{definition}

A pointwise merger uses only the current vector $(M_t^1,\ldots,M_t^K)$ and not the earlier paths. It is \emph{admissible} if no other pointwise merger dominates it everywhere and is strictly larger somewhere. The next theorem identifies this process-level class with ordinary arbitrary-dependence $e$-merging.

\begin{theorem}[Static-to-any-time transfer for pointwise mergers]\label{thm:process-merger-equivalence}
For an increasing Borel map $F:[0,\infty)^K\to[0,\infty)$, the following are equivalent:
\begin{enumerate}
\item[(i)] $F$ is an arbitrary-dependence $e$-merging function;
\item[(ii)] $F$ is a pointwise $e$-process merger.
\end{enumerate}
Moreover, admissibility is the same under the two definitions. Combining this equivalence with the fixed-dimensional admissible-merger characterization of \citet{Wang2025Merging} (see also the simpler proof of \citet{Clerico2026Merging}), $F$ is an admissible pointwise $e$-process merger if and only if there are coefficients $\lambda_0,\lambda_1,\ldots,\lambda_K\geq0$ with $\sum_{k=0}^K\lambda_k=1$ such that
\begin{equation}\label{eq:wang-process-form}
  F(x_1,\ldots,x_K)=\lambda_0+\sum_{k=1}^K\lambda_kx_k.
\end{equation}
\end{theorem}

\proofidea
At any global stopping time, the coordinates of a tuple of $e$-processes are arbitrarily dependent $e$-values, so every ordinary $e$-merger is a pointwise process merger. Conversely, arbitrary $e$-values can be embedded in processes that equal one at time zero and are constant thereafter. The two merger classes are therefore identical, and Wang's theorem gives the affine form. The complete proof is in Appendix~\ref{app:proof-process-merger}.

\begin{remark}[Pointwise versus history-dependent merging]\label{rem:process-merger-scope}
Optional stopping does not change the class of pointwise mergers when all coordinates are valid in the same global filtration. A history-dependent merger would instead use the paths $((M_u^k)_{u\leq t})_{k\leq K}$. However, the sequential one-step $e$-value setting of \citet{VovkWang2024Sequential} has a different input structure. Characterization of history-dependent mergers of parallel evolving $e$-processes is another direction for future work.
\end{remark}

\subsection{Neutral padding and projectivity}\label{sec:weighted-construction}

Suppose $M_i=(M_{i,t})_{t\geq0}$ is a global $e$-process for $H_i$, and let $a_i$ be the activation time of hypothesis $i$. We use the dormant convention
\begin{equation}\label{eq:dormant}
  M_{i,t}=1\qquad(t<a_i).
\end{equation}
The value one is neutral evidence. For finite $B\subseteq S$, define the neutral-padding embedding
\begin{equation}\label{eq:padding-map}
  \iota_{B,S}(x_B):=(x_B,\mathbf1_{S\setminus B}).
\end{equation}

\begin{definition}[Neutral-padding projective merger family]\label{def:projective-merger}
A \emph{neutral-padding projective pointwise $e$-process merger family} is a family
\[
  F_S:[0,\infty)^S\to[0,\infty),\qquad |S|<\infty,
\]
with $F_\varnothing=1$, such that each nonempty $F_S$ is a pointwise $e$-process merger and, for every finite $B\subseteq S$,
\begin{equation}\label{eq:merger-projective}
  F_S\bigl(\iota_{B,S}(x_B)\bigr)=F_B(x_B).
\end{equation}
A family is only \emph{future-extension coherent} if the equality in Equation~\eqref{eq:merger-projective} is replaced by $\geq$.
\end{definition}

The value one is neutral evidence for $e$-values \citep{VovkWang2021}. Neutral padding means adjoining coordinates fixed at one when comparing merger rules at different horizons. We use \emph{projective} in the sense of consistency across finite-dimensional systems; here the consistency maps are the embeddings $\iota_{B,S}$. Equality under these embeddings is stronger than the one-sided no-reversal condition of online closure.

\begin{remark}[Neutral entry and open-minded Bayesianism]\label{rem:open-minded-neutrality}
The neutral-padding convention is analogous to the ``forward-looking'' treatment of newly proposed hypotheses in \citet{SterkenburgDeHeide2022}. There, a new Bayesian hypothesis is assigned the learner's own past predictive distribution as its likelihood for past observations, so its introduction does not change the assessment of those observations. Here a dormant testing coordinate is assigned the neutral $e$-value one, so adding it does not change the current intersection certificate. The mathematical objects and guarantees are different, but in both cases a newly introduced hypothesis enters neutral with respect to past data.
\end{remark}

\begin{remark}[Dormant hypotheses and sleeping experts]\label{rem:sleeping-experts}
Sleeping or specialist experts in online learning can be inactive on selected rounds \citep{FreundEtAl1997Specialists,KoolenVanErven2010}. A dormant hypothesis plays a similar bookkeeping role before activation, represented here by the value one. But futher it is very different: sleeping-expert methods control regret, whereas the present construction concerns validity of evidence and multiple-testing decisions.
\end{remark}

\begin{proposition}[Projective mergers generate dynamic intersection collections]\label{prop:projective-merger}
Let $(F_S)$ be a neutral-padding projective pointwise $e$-process merger family and let the coordinate processes satisfy Equation~\eqref{eq:dormant}. Then
\begin{equation}\label{eq:projective-merger-process}
  E_t^S:=F_S\bigl((M_{i,t})_{i\in S}\bigr)
\end{equation}
forms a projectively consistent dynamic intersection $e$-process collection.
\end{proposition}

\begin{proof}
Under $P\in H_S$, every coordinate $M_i$, $i\in S$, is a $P$-$e$-process in the common global filtration. The process-merger property gives that $E^S$ is an $e$-process. If $B=S\cap I_t$, every coordinate in $S\setminus B$ is dormant and equals one, so Equation~\eqref{eq:merger-projective} gives $E_t^S=E_t^B$.
\end{proof}

The arithmetic mean is admissible at every fixed dimension but is not projective across dimensions:
\[
  F_{\{1\}}(4)=4,
  \qquad
  F_{\{1,2\}}(4,1)=\frac52<4.
\]
Thus ordinary closed-eBH intersection averages can lose evidence when a dormant future hypothesis is appended.

\subsection{Cross-horizon rigidity}\label{sec:rigidity}

\begin{theorem}[Cross-horizon rigidity of admissible process mergers]\label{thm:coherent-merger-rigidity}
Let $(F_S)$, indexed by the finite nonempty sets $S\subseteq\Iuniv$, be a family such that every $F_S$ is an admissible pointwise $e$-process merger. Suppose the family is future-extension coherent under neutral padding:
\begin{equation}\label{eq:merger-coherence}
  F_B(x_B)\leq F_S\bigl(\iota_{B,S}(x_B)\bigr)
  \qquad\text{for every finite nonempty }B\subseteq S\text{ and every }x_B\geq0.
\end{equation}
Then there is a unique family of weights $(w_i)_{i\in\Iuniv}$ satisfying
\begin{equation}\label{eq:global-weight-budget}
  w_i\geq0,\qquad \sum_{i\in\Iuniv}w_i\leq1,
\end{equation}
such that, for every finite nonempty $S$,
\begin{equation}\label{eq:affine-projective}
  F_S(x_S)=1+\sum_{i\in S}w_i(x_i-1).
\end{equation}
In particular, Equation~\eqref{eq:merger-coherence} necessarily holds with equality: every admissible coherent family is neutral-padding projective. The weights are identified by the singleton merger rules; for example,
\begin{equation}\label{eq:recover-weights}
  w_i=F_{\{i\}}(2)-1=1-F_{\{i\}}(0).
\end{equation}
\end{theorem}

\proofidea
By Wang's fixed-dimensional characterization, each $F_S$ has an affine representation with horizon-specific weights $w_{i,S}$ \citep{Wang2025Merging}. Compare this representation with that of a neutral-padded subset. Sending one coordinate to infinity makes each coefficient difference nonnegative; setting all subset coordinates to zero forces the sum of those differences to be nonpositive. Hence every difference is zero, so the weight of coordinate $i$ is the same at every horizon containing it. The complete coefficient argument is in Appendix~\ref{app:proof-rigidity}.

\begin{remark}[Cross-horizon interpretation]\label{rem:wang-comparison}
Wang's theorem gives coefficients $w_{i,S}$ separately for each finite set $S$. Theorem~\ref{thm:coherent-merger-rigidity} shows that coherence forces
\[
  w_{i,S}=w_i\qquad\text{whenever }i\in S,
\]
with $\sum_{i\in\Iuniv}w_i\leq1$. Thus the same coefficient is used for coordinate $i$ at every horizon, and the one-sided coherence inequality becomes equality under neutral padding.
\end{remark}

The singleton rules determine the weights through Equation~\eqref{eq:recover-weights}. Conversely, any nonnegative sequence with total mass at most one defines a projective family through Equation~\eqref{eq:affine-projective}.

The weights therefore cannot be renormalized over the currently active hypotheses when a new hypothesis arrives: renormalization would change the coefficients of earlier coordinates. Under symmetry the restriction is stronger still; Corollary~\ref{cor:symmetric-no-go} shows that on a countably infinite universe only the trivial merger remains. Avoiding the global weight budget requires leaving at least one assumption of Theorem~\ref{thm:coherent-merger-rigidity}, for example by using additional dependence structure or direct intersection evidence.

\begin{corollary}[The globally weighted projective collection]\label{cor:weighted-projective}
Let $M_i$ be global coordinate $e$-processes with dormant value one before activation, and let $(w_i)$ satisfy Equation~\eqref{eq:global-weight-budget}. Then
\begin{equation}\label{eq:weighted-system}
  E_t^{S,w}=1+\sum_{i\in S}w_i(M_{i,t}-1)
\end{equation}
is a neutral-padding projective dynamic intersection $e$-process collection under arbitrary cross-stream dependence.
\end{corollary}

\justification
Neutral coordinates contribute zero to Equation~\eqref{eq:weighted-system}, and the total weight bound makes each finite-dimensional affine merger valid under arbitrary dependence. Proposition~\ref{prop:projective-merger} then gives projective consistency. The complete proof is in Appendix~\ref{app:proof-weighted}.

\begin{example}[Horizon effect with three hypotheses]\label{ex:three-hypothesis-projective}
Take $\alpha=0.2$ and global weights $w_1=1/2$, $w_2=1/4$, $w_3=1/8$. Suppose the current evidence is $(M_1,M_2,M_3)=(9,1,1)$, where hypotheses 2 and 3 are still dormant. The projective certificate for every padded intersection containing hypothesis 1 is
\[
  1+\frac12(9-1)=5.
\]
Hence the singleton rejection $R=\{1\}$ meets the critical constraint $\FDP_S(R)=1\leq\alpha E^S=1$ for every such intersection, independently of whether the horizon is written as $\{1\}$, $\{1,2\}$, or $\{1,2,3\}$.

Arithmetic-mean certificates equal $9$, $5$, and $11/3$ at the same three horizons. The current rejection is certified at the first two horizons and fails the three-hypothesis constraint because $\alpha(11/3)<1$. Each arithmetic mean is a valid fixed-dimensional merger; the discrepancy comes from changing the horizon. If later $M_2$ rises to $5$, the projective certificate becomes $1+\tfrac12(8)+\tfrac14(4)=6$ while retaining the contribution already assigned to hypothesis 1.
\end{example}

The price of projectivity is that the same total weight budget must be shared across the entire hypothesis universe, including hypotheses that have not yet arrived. At a fixed dimension the weights can instead be chosen afresh for that dimension. Theorem~\ref{thm:coherent-merger-rigidity} shows that a global budget is unavoidable within the admissible pointwise merger class under arbitrary dependence.

\subsection{Persistent projective collections by adjustment}\label{sec:adjuster}

Projectivity concerns hypothesis horizons and does not prevent an $e$-process from decreasing over evidence time. To obtain persistence one may calibrate its running maximum \citep{DawidEtAl2011Evidence,TavyrikovGoemanDeHeide2026}. If the intersection processes are already time-monotone, no adjustment is needed. Otherwise, let $A:[1,\infty]\to[0,\infty]$ be a nondecreasing adjuster satisfying
\begin{equation}\label{eq:adjuster-condition}
  \int_1^\infty\frac{A(x)}{x^2}\,dx\leq1.
\end{equation}
For an $e$-process $E$, the adjusted running maximum
\begin{equation}\label{eq:adjusted-process}
  \widetilde E_t:=A\!\left(1\vee\sup_{0\leq u\leq t}E_u\right)
\end{equation}
is a pathwise nondecreasing $e$-process \citep{DawidEtAl2011Evidence}. Appendix~\ref{sup:nonadjuster} shows that time-monotone $e$-processes need not arise from this construction.

The same adjustment can lift validity to a larger filtration. Suppose $M_i$ is an $e$-process in $(\mathcal G^i_t)$ with $\mathcal G^i_t\subseteq\F_t$. Validity in $(\mathcal G^i_t)$ alone does not justify stopping with respect to $(\F_t)$. The lifting theorem of \citet{ChoeRamdas2026} shows that
\[
  L_{i,t}:=A\!\left(1\vee\sup_{0\leq u\leq t}M_{i,u}\right)
\]
is a pathwise nondecreasing $e$-process in $(\F_t)$. It may therefore be evaluated at arbitrary almost surely finite global $(\F_t)$-stopping times.

Filtration lifting changes the dormant value used by the projective construction. Let $a:=A(1)$. If $a=1$, a dormant coordinate remains at the neutral value one and no recentering is needed. Suppose $a<1$. Then a dormant coordinate is mapped to $a$ rather than one, so direct substitution into Corollary~\ref{cor:weighted-projective} would destroy neutral padding by one. Recenter and rescale the lifted coordinate by
\[
  D_{i,t}:=\frac{L_{i,t}-a}{1-a}.
\]
Since $L_{i,t}\geq a$, the process $D_i$ is nonnegative and time-monotone. Moreover, for every $P\in H_i$ and every global stopping time $\tau$,
\[
  \E_P[D_{i,\tau}]
  =\frac{\E_P[L_{i,\tau}]-a}{1-a}\leq1.
\]
Thus $D_i$ is a global $e$-process, and $D_{i,t}=0$ while hypothesis $i$ is dormant. For nonnegative weights satisfying Equation~\eqref{eq:global-weight-budget},
\[
  E_t^S:=\sum_{i\in S}w_iD_{i,t}
\]
forms a global, time-monotone intersection $e$-process collection: under $P\in H_S$ and for every global stopping time $\tau$,
\[
  \E_P[E_\tau^S]\leq\sum_{i\in S}w_i\leq1.
\]
With the bookkeeping value $E_t^\varnothing=0$, dormant coordinates contribute zero, so the collection is projectively consistent. This is a zero-neutral construction and lies outside the neutral-padding-by-one admissible class of Section~\ref{sec:rigidity}. If $\sum_{i\in S}w_i<1$, adding the unused constant mass gives a larger valid fixed-dimensional affine merger, so the zero-neutral construction is conservative. When global-filtration validity can be verified directly, as in Section~\ref{sec:shared-control}, this lifting step is unnecessary.

\begin{proposition}[Persistent projective adjustment]\label{prop:projective-adjustment}
Let $\mathbf E$ be a projectively consistent dynamic intersection $e$-process collection, including its bookkeeping extension $E^\varnothing$. For every finite $S$, define
\[
  \widetilde E_t^S
  :=A\!\left(1\vee\sup_{0\leq u\leq t}E_u^S\right).
\]
For nonempty $S$, $\widetilde E^S$ is a time-monotone $e$-process, and the full collection $\widetilde{\mathbf E}$ is projectively consistent. Consequently, its dynamic closure is setwise persistent and controls simultaneous SupFDR.
\end{proposition}

\proofidea
The adjuster theorem gives the $e$-process property and time monotonicity for every nonempty $S$. Projective equality holds at every time before $t$, so the running maxima for $S$ and $S\cap I_t$ agree. Theorem~\ref{thm:forward-persistent} then applies. The complete proof is in Appendix~\ref{app:proof-adjustment}.

\begin{corollary}[Symmetry is impossible on an open-ended horizon]\label{cor:symmetric-no-go}
Assume the hypotheses form a countably infinite universe. Under the assumptions of Theorem~\ref{thm:coherent-merger-rigidity}, suppose in addition that every $F_S$ is symmetric in its coordinates. Then
\[
  F_S(x_S)=1
  \qquad\text{for every finite nonempty }S.
\]
Thus there is no nontrivial symmetric, admissible, arbitrary-dependence, pointwise merger family that is coherent over an unbounded hypothesis horizon. If instead the universe has a known finite size $M$, symmetry forces $w_i=c$ for every coordinate, with $0\leq c\leq1/M$.
\end{corollary}

\begin{proof}
For any two distinct indices $i$ and $j$, symmetry of $F_{\{i,j\}}$ and Equation~\eqref{eq:affine-projective} imply $w_i=w_j$. Hence all global weights equal a common value $c\geq0$. On a countably infinite universe the summability condition in Equation~\eqref{eq:global-weight-budget} forces $c=0$, giving $F_S\equiv1$. If the universe has size $M$, the same budget gives $Mc\leq1$.
\end{proof}

On a countably infinite universe, symmetry therefore forces every coordinate weight to zero. A nontrivial symmetric construction on an open-ended horizon must fall outside the class covered by Theorem~\ref{thm:coherent-merger-rigidity}.

\subsection{Beyond pointwise arbitrary-dependence mergers}\label{sec:other-mergers}

The rigidity theorem applies only to pointwise mergers that are valid under arbitrary dependence. Product mergers are neutral under padding by one but require additional assumptions, such as conditional independence or a direct joint-supermartingale argument. Direct intersection betting strategies and history-dependent mergers can also use information excluded by the pointwise formulation. This is left for future work.

\section{A global-filtration construction under shared control}\label{sec:construction}

\subsection{Gaussian coordinate \texorpdfstring{$e$}{e}-processes}\label{sec:shared-control}

Following standard likelihood-ratio betting constructions \citep{ShaferEtAl2011TestMartingales,RamdasWangBook}, consider treatment arms with deterministic activation times $a_i$. At global update $t$, let a fresh control outcome $C_t\sim N(0,1)$ be observed. For every active arm $i$ with $a_i\leq t$, let $X_{i,t}\sim N(\mu_i,1)$. Conditional on $\F_{t-1}$, assume that $C_t$ and the active-arm outcomes are mutually independent with the stated Gaussian laws, and that all active comparisons use the same $C_t$. Put
\begin{equation}\label{eq:shared-z}
  Z_{i,t}:=\frac{X_{i,t}-C_t}{\sqrt2}.
\end{equation}
Then $Z_{i,t}\sim N(\mu_i/\sqrt2,1)$, while $\operatorname{Corr}(Z_{i,t},Z_{j,t})=1/2$ for distinct active arms because the control is shared.

For the one-sided null $H_i=\{\mu_i\leq0\}$ and a fixed $\lambda_i>0$, define
\begin{equation}\label{eq:shared-eprocess}
  M_{i,t}:=
  \prod_{u=a_i}^{t}
  \exp\left(\lambda_i Z_{i,u}-\frac{\lambda_i^2}{2}\right),
  \qquad M_{i,t}=1\quad(t<a_i).
\end{equation}

\begin{proposition}[Global validity with a shared control]\label{prop:shared-control}
For every $i$, $(M_{i,t})_{t\geq0}$ is a nonnegative test supermartingale for $H_i$ relative to the global filtration generated by all control outcomes, all active-arm outcomes, and all past decisions. Hence it is a global $e$-process. The coordinate processes may be strongly dependent across arms.
\end{proposition}

\proofidea
Conditional on the global past, the next treatment--control contrast is Gaussian with variance one and null mean at most zero. Its exponential likelihood-ratio increment therefore has conditional expectation at most one. Multiplying these increments gives a test supermartingale in the global filtration despite dependence across arms through the shared control. The calculation is in Appendix~\ref{app:proof-shared-control}.

Thus shared-control dependence does not prevent global validity: for each coordinate, the next multiplicative factor has conditional expectation at most one given the complete past. The example shows how the global-filtration condition can be checked directly without requiring independence across coordinate processes.

\section{Batch any-time benchmarks: closed eBH and closed BY}\label{sec:applications}

When the family is fixed, $I_t=[m]$, cross-horizon compatibility is irrelevant. Arithmetic means of coordinate $e$-processes are valid intersection processes under arbitrary dependence and yield closed eBH. 

\subsection{Stopped closed eBH}\label{sec:closed-ebh}

Let the active family be fixed, $I_t=[m]$ for all $t$, and let $M_i=(M_{i,t})_{t\geq0}$ be a global $e$-process for $H_i$. For every nonempty $S\subseteq[m]$, define the arithmetic-mean intersection process
\begin{equation}\label{eq:cebh-intersection}
  E_t^{S,\ceBH}:=\frac1{|S|}\sum_{i\in S}M_{i,t}.
\end{equation}
Under $H_S$, every coordinate in the average is null, so Equation~\eqref{eq:cebh-intersection} is an $e$-process without any independence assumption. Define the full closed-eBH candidate family
\begin{equation}\label{eq:cebh-family}
  \Rfam_{\alpha,t}^{\ceBH}
  :=\left\{R\subseteq[m]:
  \FDP_S(R)\leq\alpha E_t^{S,\ceBH}
  \text{ for every nonempty }S\subseteq[m]\right\}.
\end{equation}
At each fixed time this is the closed-eBH family of \citet{XuEtAl2026Closure}, built from the eBH procedure of \citet{WangRamdas2022}. It contains the ordinary eBH rejection set and can contain strictly larger or alternative candidate sets.

\begin{corollary}[Stopped closed eBH]\label{cor:stopped-cebh}
The candidate-family process $\Rfam_\alpha^{\ceBH}$ controls simultaneous stopped FDR at level $\alpha$ under arbitrary cross-hypothesis dependence. In particular, any measurable selection from the family at a global stopping time controls stopped FDR.
\end{corollary}

\justification
For every intersection $S$, the arithmetic mean of the globally valid coordinate processes is an $H_S$-$e$-process. Since the active family is fixed, Theorem~\ref{thm:forward-stopped} applies directly. The complete proof is in Appendix~\ref{app:proof-stopped-cebh}.

The coordinate processes must be valid in the filtration used for the global stopping rule. Local-stream validity alone need not survive stopping after inspecting the joint closed-eBH output \citep{WangDandapanthulaRamdas2025}. Section~\ref{sec:adjuster} gives a filtration-lifting construction that can be applied before forming the intersection averages.

\subsection{Persistent closed eBH}\label{sec:persistent-cebh}

Let $A$ be an adjuster and define coordinatewise
\begin{equation}\label{eq:adjusted-coordinate}
  \widetilde M_{i,t}
  :=A\!\left(1\vee\sup_{0\leq u\leq t}M_{i,u}\right).
\end{equation}
Each $\widetilde M_i$ is a time-monotone global $e$-process. The adjusted closed-eBH intersection collection is
\begin{equation}\label{eq:adjusted-cebh-intersection}
  \widetilde E_t^{S,\ceBH}
  :=\frac1{|S|}\sum_{i\in S}\widetilde M_{i,t}.
\end{equation}

\begin{corollary}[Persistent closed eBH]\label{cor:persistent-cebh}
The closed-eBH family generated by Equation~\eqref{eq:adjusted-cebh-intersection} is setwise persistent and controls simultaneous SupFDR at level $\alpha$. At every time it contains ordinary eBH applied to the adjusted coordinates.
\end{corollary}

\justification
The adjusted coordinates are time-monotone $e$-processes, and so are their intersection averages. Theorem~\ref{thm:forward-persistent} gives the result. Appendix~\ref{app:proof-persistent-cebh} gives the complete proof. The construction remains batch because the arithmetic means depend on the fixed family size.

Adjustment can reduce evidence, so the persistent procedure need not contain closed eBH applied to the unadjusted values at the same nominal level. Adjusting each intersection mean directly is also valid, but then the intersection processes need not retain the coordinatewise closed-eBH form.

\subsection{Other batch any-time benchmarks}\label{sec:closed-by}

Closed BY gives an arbitrary-dependence $p$-value analogue. With the intersection-size-dependent p-to-e calibrators used in fixed-time closed BY \citep{BenjaminiYekutieli2001,XuEtAl2026Closure}, nonincreasing coordinate $p$-processes yield time-monotone intersection $e$-processes; Appendix~\ref{sup:closed-by} gives the construction. Closed BH is different because its fixed-time guarantee relies on PRDS or related expectation conditions \citep{Goeman2026ClosedBH}. The next subsection shows that fixed-time PRDS need not survive global stopping.

\subsection{Why fixed-time PRDS is not enough for stopping-time BH}
\label{sec:why-not-bh}

A direct BH analogue would require a dependence condition that survives the global data-dependent choice of analysis time. The classical BH theorem uses independence or, more generally, PRDS \citep{BenjaminiHochberg1995,BenjaminiYekutieli2001}. Independence of the current $p$-value vector at every deterministic time is still insufficient after global stopping.

The following example assumes more than fixed-time PRDS: the two current $p$-values are independent at every deterministic time, and each coordinate process is valid under optional stopping. Let there be two true null hypotheses and fix $\alpha\in(0,1)$. Let $W$ take values $A,B,N$ with
\[
\Pr(W=A)=\Pr(W=B)=\frac{\alpha}{2},
\qquad
\Pr(W=N)=1-\alpha,
\]
and let $B_1,B_2$ be independent Bernoulli variables, independent of $W$, with
\[
\Pr(B_i=1)=\frac{\alpha}{2}.
\]
Set
\[
\F_0=\{\varnothing,\Omega\},\qquad
\F_1=\sigma(\one_{\{W=A\}}),\qquad
\F_2=\sigma(W),\qquad
\F_3=\sigma(W,B_1,B_2),
\]
and $\F_t=\F_3$ for $t\geq3$. Define $P_{1,0}=P_{2,0}=1$ and, at time one,
\[
(P_{1,1},P_{2,1})
=
\begin{cases}
	(\alpha/2,1),&W=A,\\
	(1,1),&W\neq A.
\end{cases}
\]
At time two, set
\[
(P_{1,2},P_{2,2})
=
\begin{cases}
	(1,\alpha/2),&W=B,\\
	(1,1),&W\neq B,
\end{cases}
\]
and at time three set
\[
P_{i,3}
=
\begin{cases}
	\alpha,&B_i=1,\\
	1,&B_i=0,
\end{cases}
\qquad i=1,2.
\]
Keep the processes constant for $t\geq3$.

At times one and two one coordinate is constant, and at time three independence follows from that of $B_1$ and $B_2$. Hence the two current $p$-values are independent, and therefore PRDS, at every deterministic time. Their full paths are dependent through $W$. Each coordinate is nevertheless valid under optional stopping. The only nontrivial crossing probabilities are
\[
\Pr\left(\inf_t P_{i,t}\leq\frac{\alpha}{2}\right)
=\frac{\alpha}{2}
\]
and
\[
\Pr\left(\inf_t P_{i,t}\leq\alpha\right)
=
\frac{\alpha}{2}+\frac{\alpha}{2}
-\frac{\alpha^2}{4}
=
\alpha-\frac{\alpha^2}{4}
\leq\alpha.
\]
Since the processes take values only in $\{\alpha/2,\alpha,1\}$ (apart from the initial value one), these two inequalities imply
\[
\Pr\left(\inf_{t\geq0} P_{i,t}\leq u\right)\leq u
\qquad\text{for every }u\in[0,1].
\]
Therefore $P_{i,\tau}$ is a valid $p$-value for every almost surely finite $(\F_t)$-stopping time $\tau$.

Now apply ordinary BH at level $\alpha$ to the two current $p$-values,
and let
\[
\tau
:=
\inf\{t\in\{1,2,3\}:
\operatorname{BH}_\alpha(P_{1,t},P_{2,t})\neq\varnothing\}
\wedge3.
\]
On $\{W=A\}$, BH rejects $H_1$ at time one because $P_{1,1}=\alpha/2$. On $\{W=B\}$, no rejection occurs at time one and BH rejects $H_2$ at time two. On $\{W=N\}$ the procedure reaches time three, where BH rejects both hypotheses whenever $B_1=B_2=1$, since then
\[
P_{1,3}=P_{2,3}=\alpha.
\]
Under the global null, every rejection is false.  Hence
\begin{align}
	\operatorname{FDR}(\tau)
	&=
	\Pr\{\operatorname{BH}_\alpha(P_{1,\tau},P_{2,\tau})
	\neq\varnothing\} \notag\\
	&=
	\Pr(W\in\{A,B\})
	+
	\Pr(W=N,B_1=B_2=1) \notag\\
	&=
	\alpha+(1-\alpha)\frac{\alpha^2}{4}
	>\alpha.
	\label{eq:bh-stopping-counterexample}
\end{align}

The stopped marginals remain valid as well.
For either coordinate,
\[
\Pr(P_{i,\tau}\leq\alpha/2)=\frac{\alpha}{2},
\qquad
\Pr(P_{i,\tau}\leq\alpha)
=
\frac{\alpha}{2}
+(1-\alpha)\frac{\alpha}{2}
=
\alpha-\frac{\alpha^2}{2}
\leq\alpha.
\]
Thus the stopped coordinates are still marginally valid $p$-values. Since BH nevertheless has FDR larger than $\alpha$ under the global null, the stopped vector cannot satisfy the PRDS condition needed for the classical BH guarantee \citep{BenjaminiYekutieli2001}.

Fixed-time independence at every deterministic time is therefore insufficient for BH after a global stopping rule. A direct sequential use of the classical theorem would require PRDS of the stopped vector for every admissible global stopping rule, or another joint sequential condition that yields the same FDR bound. This is substantially stronger than fixed-time PRDS.

For this reason we do not use ordinary BH as an any-time benchmark: the conditions needed are probably very strong and unrealistic; but this is left for future work, if it still turns out to be interesting. Procedures specifically designed for online or sequential BH-type control can impose conditions suited to that setting \citep{FischerXuRamdas2024OnlineBH}. 

\section{Necessity, canonical representations, and minimality}\label{sec:converse}

The converse direction uses the normalized-loss construction of fixed-time $e$-closure \citep{XuEtAl2026Closure}. Starting from a candidate-family process that already satisfies the target error criterion, we use its own worst-case loss as intersection evidence. The resulting dynamic closure contains the original procedure. Section~\ref{sec:minimality} shows that these canonical certificates are the smallest coherent certificates with this covering property. Proposition~\ref{prop:coherence-necessary} concerns the separate question of which assumptions are needed for the forward theorem.

\subsection{Canonical necessity representation for stopped FDR}

\begin{theorem}[Canonical stopped-FDR representation]\label{thm:converse-stopped}
Suppose $\C$ controls simultaneous stopped FDR at level $\alpha$. For every nonempty finite $S$, define
\begin{equation}\label{eq:canonical-current}
  E_t^{S,\can}:=\frac1\alpha\loss_t^S(\C)
  =\frac1\alpha\max_{R\in\C_t}\FDP_{S\cap I_t}(R).
\end{equation}
Then $\mathbf E^{\can}$ is a future-extension coherent dynamic intersection $e$-process collection and
\begin{equation}\label{eq:current-containment}
  \C_t\subseteq\Rfam_{\alpha,t}(\mathbf E^{\can})
  \qquad\text{for every }t.
\end{equation}
\end{theorem}

\proofidea
If $P\in H_S$, then $S\subseteq\calN(P)$, so the canonical $S$-loss is bounded by the true-null loss controlled by stopped FDR. Locality gives equality under future extension, and the defining maximum gives containment. The complete proof is in Appendix~\ref{app:proof-converse-stopped}.

\subsection{Canonical necessity representation for SupFDR}

\begin{theorem}[Canonical persistent representation]\label{thm:converse-persistent}
Suppose $\C$ controls simultaneous SupFDR at level $\alpha$. For every nonempty finite $S$, define
\begin{equation}\label{eq:canonical-running}
  \overline E_t^{S,\can}
  :=\frac1\alpha\runloss_t^S(\C)
  =\frac1\alpha\max_{0\leq u\leq t}\max_{R\in\C_u}
  \FDP_{S\cap I_u}(R).
\end{equation}
Then $\overline{\mathbf E}^{\can}$ is a future-extension coherent, time-monotone dynamic intersection $e$-process collection and
\begin{equation}\label{eq:running-containment}
  \C_t\subseteq\Rfam_{\alpha,t}(\overline{\mathbf E}^{\can})
  \qquad\text{for every }t.
\end{equation}
\end{theorem}

\proofidea
The running-loss process is time-monotone by construction. For $P\in H_S$, it is bounded by the true-null running loss controlled by SupFDR. Locality gives future-extension coherence, and the definition gives containment. The complete proof is in Appendix~\ref{app:proof-converse-persistent}.

The current-loss certificate records the largest configuration-wise loss of the candidate family at time $t$; the running-loss certificate records the largest such loss up to time $t$. Theorem~\ref{thm:minimality} shows that no smaller coherent certificate collection can cover the same procedure.

Together with the forward theorems, Theorems~\ref{thm:converse-stopped} and~\ref{thm:converse-persistent} complete the necessary-and-sufficient characterizations. Equality need not hold: the canonical closure can contain rejection sets that were absent from the original procedure but have no larger loss under any configuration.

\subsection{Pointwise minimality}\label{sec:minimality}

\begin{theorem}[Minimality of the canonical certificates]\label{thm:minimality}
Let $\C$ be any candidate-family process.
\begin{enumerate}
\item[(a)] If $\mathbf E$ is future-extension coherent and $\C_t\subseteq\Rfam_{\alpha,t}(\mathbf E)$ for every $t$, then
\begin{equation}\label{eq:min-current}
  E_t^{S,\can}\leq E_t^S
\end{equation}
for every finite nonempty $S$ and $t$.
\item[(b)] If, in addition, $\mathbf E$ is time-monotone, then
\begin{equation}\label{eq:min-running}
  \overline E_t^{S,\can}\leq E_t^S
\end{equation}
for every finite nonempty $S$ and $t$.
\end{enumerate}
\end{theorem}

\proofidea
Closure containment gives $\FDP_{S\cap I_t}(R)\leq\alpha E_t^S$ for every original candidate $R$. Maximizing over candidates yields the current canonical bound. If $\mathbf E$ is time-monotone, the same inequality at all earlier times yields the running canonical bound. The complete proof is in Appendix~\ref{app:proof-minimality}.

The normalized-loss construction comes from the fixed-time representation of \citet{XuEtAl2026Closure}. The dynamic statement is pointwise: among future-extension coherent collections whose closure contains $\C$, $E^{S,\can}_t$ is smallest for every $S,t$; among collections that are also time-monotone, $\overline E^{S,\can}_t$ is smallest. This is a minimality statement about certificates needed to cover a fixed procedure, not about the power of the procedure itself.

\subsection{The canonical hull}\label{sec:closure-operator}

Applying the forward closure map to the canonical certificates enlarges the original candidate family only by sets whose configuration-wise loss is already covered. Define the current and running hulls by
\begin{align}
  \Hh_t(\C)
  &:=\left\{R\subseteq I_t:
  \FDP_S(R)\leq\loss_t^S(\C)
  \text{ for every }S\subseteq I_t\right\},
  \label{eq:current-hull}\\
  \Hhsup_t(\C)
  &:=\left\{R\subseteq I_t:
  \FDP_S(R)\leq\runloss_t^S(\C)
  \text{ for every }S\subseteq I_t\right\}.
  \label{eq:running-hull}
\end{align}
Both maps are extensive, monotone, and idempotent; their fixed points are the loss-saturated candidate families. Appendix~\ref{sup:finite-example} gives a finite example and Appendix~\ref{sup:closure-proof} proves the closure-operator properties. \citet{SunWang2026Admissibility} develop a fixed-dimensional complete-class theory for weighted-mean closed-eBH procedures; a corresponding dynamic admissibility theory is left for future work.

\section{Discussion}\label{sec:discussion}

Dynamic $e$-closure combines a growing hypothesis family with continuing evidence updates. Future-extension coherence makes current decisions invariant to finite future horizons and, through localization, reduces the random active true-null intersection at a stopping time to a fixed terminal intersection. Time monotonicity adds SupFDR control and setwise persistence. The closure and representation arguments also apply to bounded configuration-monotone local losses.

For pointwise mergers, optional-stopping validity introduces no new fixed-dimensional merger class: it is equivalent to ordinary arbitrary-dependence $e$-merging. Cross-horizon coherence is restrictive. Together with Wang's characterization, it forces one globally summable weight sequence and exact neutral-padding projectivity. Arithmetic means remain available for fixed families, but an open-ended coherent construction in this admissible pointwise class must use a global weight budget.

The converse results show that every procedure satisfying the target error criterion is contained in a closure generated by its normalized current or running loss. These canonical certificates are pointwise minimal among coherent collections covering the procedure, and their forward closures define the current and running hulls.

The shared-control example verifies the common-filtration condition directly: each coordinate likelihood-ratio process is a test supermartingale in the filtration generated by all streams, even though the coordinates are dependent through the common control.

\subsection{Limitations}\label{sec:limitations}

The theory is developed for deterministic finite active sets and assumes that the relevant coordinate or intersection processes are valid in the global filtration used for stopping. The latter condition is essential for the stopping-time guarantees, but obtaining such globally valid processes may require either a direct model-specific argument or a potentially conservative filtration-lifting step. The general closure is also computationally demanding, since in principle it requires constraints for all subsets of the active family.

Some of the structural results have a narrower scope than the closure theory itself. In particular, the rigidity theorem concerns admissible pointwise mergers that are valid under arbitrary dependence; it does not rule out more flexible constructions that exploit dependence structure, past trajectories, or direct intersection evidence. Similarly, the shared-control construction is intended as an illustration of global-filtration validity rather than a full platform-trial model: it assumes a simple Gaussian setting with deterministic arm activation and does not cover delayed outcomes, time trends, nonconcurrent controls, or regulatory decision rules. Finally, setwise persistence concerns the family of certified rejection sets and does not by itself impose nested reporting decisions.

\subsection{Future work}\label{sec:future-work}

Several extensions follow naturally from these limitations. First, allowing random predictable hypothesis arrivals requires a different localization argument, because for random $I_T$ the terminal true-null set $\calN(P)\cap I_T$ is no longer a fixed intersection. More realistic platform-trial constructions could combine such arrivals with delayed observations, time trends, nonconcurrent controls, and more general stopping or reporting rules.

Second, the merger theory could be extended beyond the pointwise arbitrary-dependence class. Relevant directions include mergers that exploit structured dependence, direct intersection betting, and history-dependent mergers of parallel evolving $e$-processes. For persistent inference, adjusted running maxima provide one general construction, but more powerful ways of constructing time-monotone $e$-processes and calibrating persistent evidence remain to be developed.

Two further questions concern statistical and computational efficiency. The general closure can require exponentially many subset constraints, so scalable algorithms are needed for practically important subclasses. It would also be useful to develop a dynamic admissibility or complete-class theory analogous to the available fixed-dimensional results. Finally, Section~\ref{sec:why-not-bh} shows that fixed-time independence or PRDS is insufficient for ordinary BH after global stopping. Characterizing useful joint sequential conditions under which a stopped BH procedure remains valid (and whether that it is too strong to be useful) is an open question.

\section*{AI-assisted editing statement}
OpenAI's GPT-5.5 Instant  was used during development of the manuscript for language and exposition, generating and refining illustrative examples, assistance with the presentation of proofs, literature searches, and checks of LaTeX cross-references. In particular, the author did many iterations of requesting a thorough review of the draft and acted upon several of the issues that were pointed out by the LLMh. The author remains fully responsible for the content.

\section*{Funding}
This work was supported by NWO Veni grant number VI.Veni.222.018.

\clearpage
\phantomsection
\addcontentsline{toc}{section}{References}
\bibliographystyle{plainnat}
\bibliography{dynamic_eclosure}

@article{ChoeRamdas2026,
	title={Combining evidence across filtrations},
	author={Choe, Yo Joong and Ramdas, Aaditya},
	journal={Journal of the Royal Statistical Society Series B: Statistical Methodology},
	pages={qkag058},
	year={2026},
	publisher={Oxford University Press UK}
}

@article{FischerBofillBrannath2024,
	title={The online closure principle},
	author={Fischer, Lasse and Bofill Roig, Marta and Brannath, Werner},
	journal={The Annals of Statistics},
	volume={52},
	number={2},
	pages={817--841},
	year={2024},
	publisher={Institute of Mathematical Statistics}
}

@misc{FischerRamdas2024,
	title={Admissible online closed testing must employ e-values},
	author={Fischer, Lasse and Ramdas, Aaditya},
	journal={arXiv preprint arXiv:2407.15733},
	year={2024},
	note         = {Version dated 15 December 2025}
}

@misc{FischerXuRamdas2024OnlineBH,
	title={An online generalization of the (e-) Benjamini-Hochberg procedure},
	author={Fischer, Lasse and Xu, Ziyu and Ramdas, Aaditya},
	journal={arXiv preprint arXiv:2407.20683},
	year={2024},
	note = {Version dated 26 February 2026}
}

@misc{Goeman2026ClosedBH,
	title={A Uniform Improvement of the Benjamini-Hochberg Procedure via e-Closure},
	author={Goeman, Jelle},
	journal={arXiv preprint arXiv:2606.01854},
	year={2026},
	note         = {Version dated 10 June 2026}
}

@article{GoemanHemerikSolari2021,
	title={Only closed testing procedures are admissible for controlling false discovery proportions},
	author={Goeman, Jelle J and Hemerik, Jesse and Solari, Aldo},
	journal={The Annals of Statistics},
	volume={49},
	number={2},
	pages={1218--1238},
	year={2021},
	publisher={JSTOR}
}

@article{GoemanSolari2011,
  title={Multiple testing for exploratory research},
author={Goeman, Jelle J and Solari, Aldo},
journal={Statistical Science},
pages={584--597},
year={2011},
publisher={JSTOR}
}

@article{LeeEtAl2021Platform,
  title={Statistical consideration when adding new arms to ongoing clinical trials: the potentials and the caveats},
author={Lee, Kim May and Brown, Louise C and Jaki, Thomas and Stallard, Nigel and Wason, James},
journal={Trials},
volume={22},
number={1},
pages={203},
year={2021},
publisher={Springer}
}

@article{MarcusPeritzGabriel1976,
  title={On closed testing procedures with special reference to ordered analysis of variance},
author={Marcus, Ruth and Eric, Peritz and Gabriel, K Ruben},
journal={Biometrika},
volume={63},
number={3},
pages={655--660},
year={1976},
publisher={Oxford University Press}
}

@misc{RamdasRufLarssonKoolen2020,
  title={Admissible anytime-valid sequential inference must rely on nonnegative martingales},
author={Ramdas, Aaditya and Ruf, Johannes and Larsson, Martin and Koolen, Wouter},
journal={arXiv preprint arXiv:2009.03167},
year={2020},
  note         = {Version dated 5 November 2022}
}

@article{RobertsonWasonRamdas2023,
  title={Online multiple hypothesis testing},
author={Robertson, David S and Wason, James MS and Ramdas, Aaditya},
journal={Statistical science: a review journal of the Institute of Mathematical Statistics},
volume={38},
number={4},
pages={557},
year={2023}
}

@article{SterkenburgDeHeide2022,
  title={On the truth-convergence of open-minded Bayesianism},
author={Sterkenburg, Tom F and De Heide, Rianne},
journal={The Review of Symbolic Logic},
volume={15},
number={1},
pages={64--100},
year={2022},
publisher={Cambridge University Press}
}

@misc{SunWang2026Admissibility,
  title={Admissibility and Complete Classes for False Discovery Rate Control with E-values},
author={Sun, Liulei and Wang, Ruodu},
journal={arXiv preprint arXiv:2607.14380},
year={2026},
note = {Version dated 15 July 2026}
}

@article{TavyrikovGoemanDeHeide2026,
  title={Carefree multiple testing with e-processes},
author={Tavyrikov, Yury and Goeman, Jelle J and de Heide, Rianne},
journal={Electronic Journal of Statistics},
volume={20},
number={2},
pages={3178--3189},
year={2026},
publisher={The Institute of Mathematical Statistics and the Bernoulli Society}
}

@article{VovkWang2021,
  title={E-values: Calibration, combination and applications},
author={Vovk, Vladimir and Wang, Ruodu},
journal={The Annals of Statistics},
volume={49},
number={3},
pages={1736--1754},
year={2021},
publisher={Institute of Mathematical Statistics}
}

@article{VovkWang2024Sequential,
  title={Merging sequential e-values via martingales},
author={Vovk, Vladimir and Wang, Ruodu},
journal={Electronic Journal of Statistics},
volume={18},
number={1},
pages={1185--1205},
year={2024},
publisher={The Institute of Mathematical Statistics and the Bernoulli Society}
}

@article{Wang2025Merging,
  title={The only admissible way of merging arbitrary e-values},
author={Wang, Ruodu},
journal={Biometrika},
volume={112},
number={2},
pages={asaf020},
year={2025},
publisher={Oxford University Press}
}

@article{WangDandapanthulaRamdas2025,
  title={Anytime-valid FDR control with the stopped e-BH procedure},
author={Wang, Hongjian and Dandapanthula, Sanjit and Ramdas, Aaditya},
journal={Statistics \& Probability Letters},
pages={110512},
year={2025},
publisher={Elsevier}
}

@misc{XuEtAl2026Closure,
  title={Bringing closure to false discovery rate control: A general principle for multiple testing},
author={Xu, Ziyu and Solari, Aldo and Fischer, Lasse and de Heide, Rianne and Ramdas, Aaditya and Goeman, Jelle},
journal={arXiv preprint arXiv:2509.02517},
year={2025},
  note         = {Version dated 30 July 2026}
}

@misc{XuFischerRamdas2026Online,
  title={Improving online FDR procedures via online analogs of e-closure and compound e-values},
author={Xu, Ziyu and Fischer, Lasse and Ramdas, Aaditya},
journal={arXiv preprint arXiv:2603.24792},
year={2026},
  note         = {Version dated 8 July 2026}
}

@article{BenjaminiHochberg1995,
  title={Controlling the false discovery rate: a practical and powerful approach to multiple testing},
author={Benjamini, Yoav and Hochberg, Yosef},
journal={Journal of the Royal statistical society: series B (Methodological)},
volume={57},
number={1},
pages={289--300},
year={1995},
publisher={Wiley Online Library}
}

@article{BenjaminiYekutieli2001,
  title={The control of the false discovery rate in multiple testing under dependency},
author={Benjamini, Yoav and Yekutieli, Daniel},
journal={Annals of statistics},
pages={1165--1188},
year={2001},
publisher={JSTOR}
}

@article{WangRamdas2022,
  title={False discovery rate control with e-values},
author={Wang, Ruodu and Ramdas, Aaditya},
journal={Journal of the Royal Statistical Society Series B: Statistical Methodology},
volume={84},
number={3},
pages={822--852},
year={2022},
publisher={Oxford University Press}
}

@article{DawidEtAl2011Evidence,
  title={Insuring against loss of evidence in game-theoretic probability},
author={Dawid, A Philip and de Rooij, Steven and Shafer, Glenn and Shen, Alexander and Vereshchagin, Nikolai and Vovk, Vladimir},
journal={Statistics \& Probability Letters},
volume={81},
number={1},
pages={157--162},
year={2011},
publisher={Elsevier}
}

@article{ShaferEtAl2011TestMartingales,
  title={Test martingales, Bayes factors and p-values},
author={Shafer, Glenn and Shen, Alexander and Vereshchagin, Nikolai and Vovk, Vladimir},
journal={Statistical Science},
pages={84--101},
year={2011},
publisher={JSTOR}
}

@misc{RamdasWangBook,
  title={Hypothesis testing with e-values},
author={Ramdas, Aaditya and Wang, Ruodu},
journal={Foundations and Trends{\textregistered} in Statistics},
volume={1},
number={1-2},
pages={1--390},
year={2025},
publisher={Emerald Publishing Limited}
}

@article{Clerico2026Merging,
  title={A simple geometric proof for the characterisation of e-merging functions},
author={Clerico, Eugenio},
journal={Statistics \& Probability Letters},
pages={110750},
year={2026},
publisher={Elsevier}
}

@article{FosterStine2008,
  title={$\alpha$-investing: a procedure for sequential control of expected false discoveries},
author={Foster, Dean P and Stine, Robert A},
journal={Journal of the Royal Statistical Society Series B: Statistical Methodology},
volume={70},
number={2},
pages={429--444},
year={2008},
publisher={Oxford University Press}
}

@article{JavanmardMontanari2018,
  title={Online rules for control of false discovery rate and false discovery exceedance},
author={Javanmard, Adel and Montanari, Andrea},
journal={The Annals of statistics},
volume={46},
number={2},
pages={526--554},
year={2018},
publisher={JSTOR}
}

@article{Gabriel1969,
  author  = {Gabriel, K. Ruben},
  title   = {Simultaneous test procedures---some theory of multiple comparisons},
  journal = {Annals of Mathematical Statistics},
  year    = {1969},
  volume  = {40},
  number  = {1},
  pages   = {224--250}
}

@inproceedings{FreundEtAl1997Specialists,
  author    = {Freund, Yoav and Schapire, Robert E. and Singer, Yoram and Warmuth, Manfred K.},
  title     = {Using and combining predictors that specialize},
  booktitle = {Proceedings of the Twenty-Ninth Annual ACM Symposium on Theory of Computing},
  year      = {1997},
  pages     = {334--343},
  publisher = {ACM},
  doi       = {10.1145/258533.258616}
}

@misc{KoolenVanErven2010,
  author       = {Koolen, Wouter M. and van Erven, Tim},
  title        = {Freezing and sleeping: Tracking experts that learn by evolving past posteriors},
  year         = {2010},
  howpublished = {arXiv:1008.4654}
}

@inproceedings{XuRamdas2022Dynamic,
  author    = {Xu, Ziyu and Ramdas, Aaditya},
  title     = {Dynamic Algorithms for Online Multiple Testing},
  booktitle = {Mathematical and Scientific Machine Learning},
  series    = {Proceedings of Machine Learning Research},
  volume    = {145},
  pages     = {955--986},
  year      = {2022},
  publisher = {PMLR}
}

@book{Ville1939,
  author    = {Ville, Jean},
  title     = {\'{E}tude critique de la notion de collectif},
  year      = {1939},
  publisher = {Gauthier-Villars},
  address   = {Paris}
}

\clearpage
\appendix

\section{Complete proofs of the main results}\label{app:proofs}

\subsection{Proof of Proposition~\ref{prop:horizon-invariance}: Horizon invariance}\label{app:proof-horizon}

\begin{proof}
The inclusion from left to right follows by taking $S\subseteq I_t$ in Equation~\eqref{eq:horizon-closure}. Conversely, let $R\in\Rfam_{\alpha,t}(\mathbf E)$ and fix $S\subseteq J$ with $B:=S\cap I_t\neq\varnothing$. Current feasibility and future-extension coherence give
\[
  \FDP_{S\cap I_t}(R)=\FDP_B(R)
  \leq\alpha E_t^B\leq\alpha E_t^S.
\]
Hence $R\in\Rfam_{\alpha,t}^{J}(\mathbf E)$. Under projective consistency $E_t^S=E_t^B$, so the two constraints coincide.
\end{proof}

\subsection{Proof of Proposition~\ref{prop:coherence-necessary}: Necessity of coherence for the forward theorem}\label{app:proof-coherence-necessary}

\begin{proof}
Let the statistical model consist of one distribution $P$ on $\Omega=\{a,b\}$ with $P(a)=P(b)=1/2$, and let both null hypotheses equal the full model. Put $I_0=\varnothing$, $I_1=\{1\}$, and $I_2=\{1,2\}$, with $I_t=I_2$ thereafter. Let $\F_0$ be trivial and $\F_t=\sigma(\{a\})$ for $t\geq1$.

For each nonempty $S\subseteq\{1,2\}$, set $E_0^S=1$ and, for $t\geq1$, define
\[
 E_t^{\{1\}}=2\one_{\{a\}},\qquad
 E_t^{\{2\}}=2\one_{\{b\}},\qquad
 E_t^{\{1,2\}}=2\one_{\{b\}}.
\]
Each process has conditional expectation one at the first update and is constant afterwards, so every $E^S$ is a nonnegative test martingale for $H_S$.

Define the global stopping time
\[
 \tau=\begin{cases}1,&\omega=a,\\2,&\omega=b.\end{cases}
\]
On $\{a\}$, at time one the set $R=\{1\}$ satisfies its only nonempty closure constraint because
\[
 \FDP_{\{1\}}(R)=1=\tfrac12 E_1^{\{1\}}.
\]
On $\{b\}$, at time two the set $R=\{2\}$ satisfies all three constraints: the $\{1\}$-constraint has loss zero, while
\[
 \FDP_{\{2\}}(R)=\FDP_{\{1,2\}}(R)=1
 =\tfrac12E_2^{\{2\}}=\tfrac12E_2^{\{1,2\}}.
\]
Thus the dynamic closure contains an all-null rejection set at the stopping time on every sample path. Hence
\[
 \max_{R\in\Rfam_{1/2,\tau}(\mathbf E)}\FDP_{\{1,2\}}(R)=1
 \quad\text{almost surely},
\]
and the stopped FDR equals one.

The failure is precisely a coherence violation. At time one on $\{a\}$,
\[
 E_1^{\{1\}}=2>0=E_1^{\{1,2\}},
\]
so the stopping rule selects favourable evidence from different intersection indices on different paths. Each deterministic-index process is valid, but the random-index process is not controlled.
\end{proof}

\subsection{Proof of Lemma~\ref{lem:random-intersection}: Random-intersection localization}\label{app:proof-localization}

\begin{proof}
Fix $P$ and a deterministic horizon $T$, and let $\sigma_T=\tau\wedge T$. If $\calN_T(P)=\varnothing$, then $Z_{\sigma_T}(P)=0$. Otherwise, on the event $\calN_{\sigma_T}(P)\neq\varnothing$,
\[
  \calN_{\sigma_T}(P)=\calN_T(P)\cap I_{\sigma_T},
\]
so future-extension coherence gives
\[
  Z_{\sigma_T}(P)\leq E_{\sigma_T}^{\calN_T(P)}.
\]
The terminal set $\calN_T(P)$ is deterministic and finite under fixed $P$, and $P\in H_{\calN_T(P)}$. Hence the $e$-process property implies
\[
  \E_P[Z_{\sigma_T}(P)]\leq1.
\]
As $T\to\infty$, $Z_{\sigma_T}(P)\to Z_\tau(P)$ almost surely because $\tau<\infty$. Fatou's lemma proves Equation~\eqref{eq:random-intersection-stopped}.

For the maximal statement, fix $T$. If $\calN_T(P)=\varnothing$, then $Z_t(P)=0$ for every $t\leq T$. Otherwise, future-extension coherence and time monotonicity imply, for every $t\leq T$,
\[
  Z_t(P)\leq E_t^{\calN_T(P)}\leq E_T^{\calN_T(P)}.
\]
Therefore
\[
  \sup_{0\leq t\leq T}Z_t(P)\leq E_T^{\calN_T(P)},
\]
whose expectation is at most one. Letting $T\to\infty$ and applying monotone convergence proves Equation~\eqref{eq:random-intersection-max}.
\end{proof}

\subsection{Proof of Theorem~\ref{thm:forward-stopped}: Dynamic \texorpdfstring{$e$}{e}-closure principle for stopped FDR}\label{app:proof-forward-stopped}

\begin{proof}
For every candidate $R\subseteq I_\tau$, locality and the closure constraint for the active true-null set give
\[
  \FDP_{\calN(P)}(R)=\FDP_{\calN_\tau(P)}(R)
  \leq \alpha Z_\tau(P),
\]
with the convention that the bound is zero when $\calN_\tau(P)=\varnothing$. Maximize over candidates, take expectations, and apply Equation~\eqref{eq:random-intersection-stopped}.
\end{proof}

\subsection{Proof of Theorem~\ref{thm:forward-persistent}: Persistent dynamic \texorpdfstring{$e$}{e}-closure principle}\label{app:proof-forward-persistent}

\begin{proof}
The closure constraint gives, for every $t$,
\[
  \max_{R\in\Rfam_{\alpha,t}(\mathbf E)}\FDP_{\calN(P)}(R)
  \leq\alpha Z_t(P).
\]
Take the supremum and use Equation~\eqref{eq:random-intersection-max}.

For persistence, let $R$ be feasible at time $t$ and fix $S\subseteq I_u$. Put $B=S\cap I_t$. If $B=\varnothing$, then $\FDP_S(R)=0$. Otherwise,
\[
  \FDP_S(R)=\FDP_B(R)
  \leq\alpha E_t^B
  \leq\alpha E_t^S
  \leq\alpha E_u^S,
\]
using current feasibility, future-extension coherence, and time monotonicity.
\end{proof}

\subsection{Proof of Theorem~\ref{thm:characterization-stopped}: Necessary and sufficient dynamic \texorpdfstring{$e$}{e}-closure for stopped FDR}\label{app:proof-char-stopped}

\begin{proof}
The implication (ii)$\Rightarrow$(i) follows from Theorem~\ref{thm:forward-stopped}, because every candidate in $\C_t$ is also a candidate in the larger closure family. The implication (i)$\Rightarrow$(ii) is Theorem~\ref{thm:converse-stopped}, proved in Appendix~\ref{app:proof-converse-stopped}.
\end{proof}

\subsection{Proof of Theorem~\ref{thm:characterization-sup}: Necessary and sufficient dynamic \texorpdfstring{$e$}{e}-closure for SupFDR}\label{app:proof-char-sup}

\begin{proof}
The implication (ii)$\Rightarrow$(i) and setwise persistence follow from Theorem~\ref{thm:forward-persistent}. The implication (i)$\Rightarrow$(ii) is Theorem~\ref{thm:converse-persistent}, proved in Appendix~\ref{app:proof-converse-persistent}.
\end{proof}

\subsection{Proof of Theorem~\ref{thm:process-merger-equivalence}}\label{app:proof-process-merger}

\begin{proof}
Suppose first that $F$ merges arbitrarily dependent $e$-values. Let $M^1,\ldots,M^K$ be $e$-processes for $H$ in one filtration. Let $\tau$ be an almost surely finite global stopping time. For every $P\in H$, each $M_\tau^k$ is an $e$-variable under $P$. Their dependence is unrestricted, so the merger property gives
\[
  \E_P\!\left[F(M_\tau^1,\ldots,M_\tau^K)\right]\leq1.
\]
The output process is adapted because $F$ is Borel. Hence it is an $H$-$e$-process.

Conversely, let $X_1,\ldots,X_K$ be arbitrary, possibly dependent, $e$-variables. On the filtration $\F_0=\{\varnothing,\Omega\}$ and $\F_t=\sigma(X_1,\ldots,X_K)$ for $t\geq1$, define
\[
  M_0^k=1,\qquad M_t^k=X_k\quad(t\geq1).
\]
Each $M^k$ is an $e$-process: because $\F_0$ is trivial, any stopping time either equals zero on every path or is at least one on every path; in the latter case $M_\tau^k=X_k$. Thus optional-stopping validity reduces to $1\leq1$ or $\E[X_k]\leq1$. Applying the process-merger property at the deterministic stopping time $1$ yields
\[
  \E[F(X_1,\ldots,X_K)]\leq1.
\]
Thus $F$ is an ordinary arbitrary-dependence $e$-merging function.

The classes in (i) and (ii) are identical and use the same pointwise domination order, so their admissible elements are identical. Equation~\eqref{eq:wang-process-form} then follows from \citet[Theorem~1]{Wang2025Merging}.
\end{proof}

\subsection{Proof of Proposition~\ref{prop:projective-merger}: Projective mergers generate dynamic intersection collections}\label{app:proof-projective-merger}

\begin{proof}
Under $P\in H_S$, every coordinate $M_i$, $i\in S$, is a $P$-$e$-process in the common global filtration. The process-merger property gives that $E^S$ is an $e$-process. If $B=S\cap I_t$, every coordinate in $S\setminus B$ is dormant and equals one, so Equation~\eqref{eq:merger-projective} gives $E_t^S=E_t^B$.
\end{proof}

\subsection{Proof of Theorem~\ref{thm:coherent-merger-rigidity}: Cross-horizon rigidity of admissible process mergers}\label{app:proof-rigidity}

\begin{proof}
By Theorem~\ref{thm:process-merger-equivalence}, each $F_S$ is an admissible ordinary $e$-merging function. Hence \citet{Wang2025Merging} gives coefficients $w_{i,S}\geq0$ satisfying $\sum_{i\in S}w_{i,S}\leq1$ and
\[
  F_S(x_S)=1+\sum_{i\in S}w_{i,S}(x_i-1).
\]
Fix finite nonempty $B\subseteq S$. Padding the coordinates in $S\setminus B$ by one and using Equation~\eqref{eq:merger-coherence} gives
\begin{equation}\label{eq:weight-difference}
  \sum_{i\in B}\bigl(w_{i,S}-w_{i,B}\bigr)(x_i-1)\geq0
  \qquad\text{for every }x_B\geq0.
\end{equation}
Put $d_i=w_{i,S}-w_{i,B}$. Setting all coordinates except $x_j$ equal to one and letting $x_j\to\infty$ shows that $d_j\geq0$ for every $j\in B$. Setting every coordinate in $B$ equal to zero gives $-\sum_{i\in B}d_i\geq0$. Hence every $d_i=0$.

If two finite sets contain $i$, comparison of each with their union shows that the two coefficients of coordinate $i$ coincide. Thus the weight attached to coordinate $i$ does not depend on the finite set containing it; call it $w_i$. Every finite partial sum is at most one, so $\sum_{i\in\Iuniv}w_i\leq1$. Equation~\eqref{eq:affine-projective} follows, and padding by ones makes every added term vanish, proving exact projectivity. Applying the formula to the singleton set $\{i\}$ at $x=2$ or $x=0$ gives Equation~\eqref{eq:recover-weights}; hence the global sequence is unique and can be read directly from the singleton rules.
\end{proof}

\subsection{Proof of Corollary~\ref{cor:weighted-projective}: The globally weighted projective collection}\label{app:proof-weighted}

\begin{proof}
For each finite $S$, Equation~\eqref{eq:weighted-system} is the admissible pointwise merger in Equation~\eqref{eq:affine-projective}. Apply Proposition~\ref{prop:projective-merger}.
\end{proof}

\subsection{Proof of Proposition~\ref{prop:projective-adjustment}: Persistent projective adjustment}\label{app:proof-adjustment}

\begin{proof}
For every nonempty $S$, the adjuster theorem gives the $e$-process property and time monotonicity. Fix finite $S$ and time $t$, and put $B=S\cap I_t$. For every $u\leq t$, we have $S\cap I_u=B\cap I_u$, so projectivity gives $E_u^S=E_u^B$, including the case $B=\varnothing$ through the bookkeeping extension. The running maxima through time $t$ are therefore equal, and hence $\widetilde E_t^S=\widetilde E_t^B$. Theorem~\ref{thm:forward-persistent} applies to the nonempty intersections.
\end{proof}

\subsection{Proof of Proposition~\ref{prop:shared-control}: Global validity with a shared control}\label{app:proof-shared-control}

\begin{proof}
	Fix a coordinate $i$ and suppose that $\mu_i\leq0$. Conditional on $\F_{t-1}$, the new observations $X_{i,t}$ and $C_t$ are independent of the past, with
	\begin{equation}
		X_{i,t}\sim N(\mu_i,1),
		\qquad
		C_t\sim N(0,1).
	\end{equation}
	Hence
	\begin{equation}
		Z_{i,t}
		=\frac{X_{i,t}-C_t}{\sqrt{2}}
		\sim N\left(\frac{\mu_i}{\sqrt{2}},1\right)
	\end{equation}
	conditionally on $\F_{t-1}$. Therefore,
	\begin{align}
		\E\left[
		\exp\left(\lambda_i Z_{i,t}-\frac{\lambda_i^2}{2}\right)
		\middle|\F_{t-1}
		\right]
		&=
		\exp\left(\frac{\lambda_i\mu_i}{\sqrt{2}}\right)
		\leq 1.
	\end{align}
	
	Although the variables $Z_{i,t}$ are dependent across coordinates because they share the same control observation $C_t$, this dependence is irrelevant for the calculation. To prove that $M_i$ is a test supermartingale, we only need the conditional expectation of the increment of coordinate $i$ given the global past. The global filtration may therefore contain the histories of all other arms without affecting the calculation above.
	
	Since $M_{i,t-1}$ is $\F_{t-1}$-measurable, for $t\geq a_i$,
	\begin{align}
		\E[M_{i,t}\mid\F_{t-1}]
		&=
		M_{i,t-1}
		\E\left[
		\exp\left(\lambda_i Z_{i,t}-\frac{\lambda_i^2}{2}\right)
		\middle|\F_{t-1}
		\right] \\
		&\leq M_{i,t-1}.
	\end{align}
	Before activation, $M_{i,t}=1$, so the supermartingale property is immediate. Thus $M_i$ is a nonnegative test supermartingale under $H_i$ in the global filtration. By optional stopping, it is therefore a global $e$-process.
\end{proof}

\subsection{Proof of Corollary~\ref{cor:stopped-cebh}: Stopped closed eBH}\label{app:proof-stopped-cebh}

\begin{proof}
For every nonempty $S$ and every $P\in H_S$,
\[
  \E_P[E_\tau^{S,\ceBH}]
  =\frac1{|S|}\sum_{i\in S}\E_P[M_{i,\tau}]\leq1
\]
for every global stopping time $\tau$. Hence Equation~\eqref{eq:cebh-intersection} is an intersection $e$-process collection. Since the active family is fixed, Theorem~\ref{thm:forward-stopped} applies directly.
\end{proof}

\subsection{Proof of Corollary~\ref{cor:persistent-cebh}: Persistent closed eBH}\label{app:proof-persistent-cebh}

\begin{proof}
Arithmetic means preserve both the $e$-process property and time monotonicity. Apply Theorem~\ref{thm:forward-persistent}. The pointwise containment of eBH \citep{WangRamdas2022} in closed eBH is the fixed-time result of \citet{XuEtAl2026Closure}, applied to the adjusted coordinate vector at each time.
\end{proof}

\subsection{Proof of Theorem~\ref{thm:converse-stopped}: Canonical stopped-FDR representation}\label{app:proof-converse-stopped}

\begin{proof}
Adaptedness and nonnegativity are immediate. Fix a nonempty finite $S$, a distribution $P\in H_S$, and an almost surely finite stopping time $\tau$. Because $S\subseteq\calN(P)$, configuration monotonicity gives
\[
  E_\tau^{S,\can}
  =\frac1\alpha\loss_\tau^S(\C)
  \leq\frac1\alpha\loss_\tau^{\calN(P)}(\C).
\]
Taking expectations and using stopped FDR proves $\E_P[E_\tau^{S,\can}]\leq1$.

For future-extension coherence, fix $S$ with $B:=S\cap I_t\neq\varnothing$. Every candidate at time $t$ is contained in $I_t$, so
\[
  \loss_t^S(\C)=\loss_t^B(\C).
\]
Thus equality holds in Equation~\eqref{eq:future-coherence}. Finally, if $R\in\C_t$, then for every nonempty $S\subseteq I_t$,
\[
  \FDP_S(R)\leq\loss_t^S(\C)=\alpha E_t^{S,\can},
\]
which proves containment.
\end{proof}

\subsection{Proof of Theorem~\ref{thm:converse-persistent}: Canonical persistent representation}\label{app:proof-converse-persistent}

\begin{proof}
Adaptedness, nonnegativity, and time monotonicity are immediate. If $S\subseteq T$, then configuration monotonicity implies $\overline E_t^{S,\can}\leq\overline E_t^{T,\can}$, so the canonical collection is fully set-monotone and hence future-extension coherent. More precisely, if $B:=S\cap I_t\neq\varnothing$, then every candidate observed by time $t$ lies in an active set contained in $I_t$, so
\[
  \overline E_t^{S,\can}=\overline E_t^{B,\can}.
\]

Fix a nonempty finite $S$, $P\in H_S$, and an almost surely finite stopping time $\tau$. Since $S\subseteq\calN(P)$,
\[
  \overline E_\tau^{S,\can}
  \leq\frac1\alpha\sup_{u\geq0}\loss_u^{\calN(P)}(\C).
\]
The right-hand side has expectation at most one by SupFDR control. Thus $\overline E^{S,\can}$ is an $e$-process for $H_S$. Containment follows because, for $R\in\C_t$ and $S\subseteq I_t$,
\[
  \FDP_S(R)\leq\loss_t^S(\C)\leq\runloss_t^S(\C)
  =\alpha\overline E_t^{S,\can}.
\]
\end{proof}

\subsection{Proof of Theorem~\ref{thm:minimality}: Minimality of the canonical certificates}\label{app:proof-minimality}

\begin{proof}
Fix a finite nonempty $S$ and a time $t$, and put $B=S\cap I_t$. If $B=\varnothing$, then $\FDP_B(R)=0\leq\alpha E_t^S$ for every $R\in\C_t$ by nonnegativity. If $B\neq\varnothing$, closure containment and future-extension coherence give
\[
  \FDP_B(R)
  \leq\alpha E_t^B
  \leq\alpha E_t^S.
\]
Thus in either case $\FDP_{S\cap I_t}(R)\leq\alpha E_t^S$. Maximizing over $R\in\C_t$ and dividing by $\alpha$ proves Equation~\eqref{eq:min-current}.

For the running statement, apply the current bound at every $u\leq t$:
\[
  \frac1\alpha\loss_u^S(\C)\leq E_u^S\leq E_t^S,
\]
where the last inequality uses time monotonicity. Maximizing over $u\leq t$ proves Equation~\eqref{eq:min-running}.
\end{proof}

\section{Auxiliary results and proofs}\label{app:auxiliary}

\subsection{A finite canonical-hull example}\label{sup:finite-example}

This example illustrates the order-theoretic hull discussed in the main text. Consider one time point with index set $I=\{1,2,3\}$ and candidate family
\begin{equation}\label{sup:eq:finite-family}
  \C=\bigl\{\{1,2\},\{1,3\},\{2,3\}\bigr\}.
\end{equation}
For a possible true-null set $S$, write
\[
  \loss^S(\C):=\max_{R\in\C}\FDP_S(R).
\]
The maximal loss depends only on $|S|$:
\begin{center}
\begin{tabular}{@{}cccl@{}}
\toprule
$|S|$ & representative & $\loss^S(\C)$ & explanation \\
\midrule
$0$ & $\varnothing$ & $0$ & no false discoveries \\
$1$ & $\{1\}$ & $1/2$ & choose a certified pair containing $1$ \\
$2$ & $\{1,2\}$ & $1$ & the pair $\{1,2\}$ is certified \\
$3$ & $I$ & $1$ & every certified pair is entirely null \\
\bottomrule
\end{tabular}
\end{center}

The canonical hull always contains the empty set, and in this example it also adds the full set $I$. The substantive addition is $I$: for $R=I$, the FDP is $1/3$ under a singleton configuration, $2/3$ under a two-element configuration, and one under $S=I$. Each value is no larger than the corresponding loss already authorized by the original family. No singleton is added: for example, $R=\{1\}$ has FDP one under $S=\{1\}$, exceeding the available bound $1/2$. Therefore
\begin{equation}\label{sup:eq:finite-hull}
  \Hh(\C)
  =\bigl\{\varnothing,\{1,2\},\{1,3\},\{2,3\},\{1,2,3\}\bigr\}.
\end{equation}
Every set added by the hull has no larger loss under any configuration, so the maximal loss profile is unchanged. Applying the hull a second time therefore adds nothing.

\subsection{Extension to bounded monotone local losses}\label{sup:losses}

The preceding proofs extend directly to the bounded local losses introduced in Section~\ref{sec:general-losses}. Let $f_S(R)\in[0,1]$ be a loss indexed by a possible true-null configuration $S$. Assume configuration monotonicity,
\begin{equation}\label{sup:eq:general-monotone}
  S\subseteq T\quad\Longrightarrow\quad f_S(R)\leq f_T(R),
\end{equation}
and action locality,
\begin{equation}\label{sup:eq:general-locality}
  f_S(R)=f_{S\cap R}(R).
\end{equation}
Define current and running loss profiles by replacing FDP with $f$, and replace the dynamic closure constraint by $f_S(R)\leq\alpha E_t^S$.

\begin{theorem}[Dynamic closure for bounded monotone local losses]\label{sup:thm:general-loss}
Under Equations~\eqref{sup:eq:general-monotone} and~\eqref{sup:eq:general-locality}, the forward stopped-loss and persistent supremum-loss principles, the canonical converse theorems, pointwise minimality, and the closure-operator results remain valid with FDP replaced by $f$.
\end{theorem}

\begin{proof}
If $R\subseteq I_t$, locality gives
\[
  f_S(R)=f_{S\cap R}(R)=f_{(S\cap I_t)\cap R}(R)=f_{S\cap I_t}(R).
\]
This is the only FDP-specific locality identity used in the forward proofs. In the converse and minimality proofs, $S\subseteq\mathcal N(P)$ and configuration monotonicity replace the corresponding monotonicity property of FDP. Since $f_S(R)\in[0,1]$, all current and running loss profiles are finite. The closure-operator arguments depend only on comparisons between these loss profiles. Thus the proofs above carry over with FDP replaced by $f$.
\end{proof}

Examples include familywise-error loss $f_S(R)=\one\{|S\cap R|>0\}$ and weighted FDP losses with nonnegative weights.

\subsection{Proof of the closure-operator theorem}\label{sup:closure-proof}

We give the details for the order-theoretic hull induced by the canonical certificates and summarized in the main text. For an arbitrary candidate process $\C$, let $\Hh(\C)$ be the current-loss hull and $\Hhsup(\C)$ the running-loss hull, and order candidate processes pointwise by inclusion.

Extensivity is immediate. If $R\in\C_t$, then for every configuration $S$ its loss is bounded by the maximum used to define the corresponding current or running profile, so $R$ belongs to both hulls.

For monotonicity, suppose $\C_t\subseteq\D_t$ for every $t$. Then $\loss_t^S(\C)\leq\loss_t^S(\D)$ for every $S,t$, and the same inequality holds for the running profiles. Every constraint defining the hull of $\D$ is therefore at least as permissive as the corresponding constraint for $\C$.

For current-hull idempotence, put $\D=\Hh(\C)$. By definition, every $R\in\D_t$ satisfies
\[
  \FDP_S(R)\leq\loss_t^S(\C)
\]
for every $S\subseteq I_t$. Hence $\loss_t^S(\D)\leq\loss_t^S(\C)$. Extensivity gives the reverse inequality because $\C_t\subseteq\D_t$. The current profiles therefore agree, and applying the current hull again changes nothing.

For the persistent hull, put $\D=\Hhsup(\C)$. At every time $u$,
\[
  \loss_u^S(\D)\leq\runloss_u^S(\C).
\]
Consequently,
\[
  \runloss_t^S(\D)
  =\max_{u\leq t}\loss_u^S(\D)
  \leq\max_{u\leq t}\runloss_u^S(\C)
  =\runloss_t^S(\C).
\]
Extensivity again gives the reverse inequality. Thus the running profiles also agree after one application, proving idempotence.

\subsection{Persistent closed BY in the batch any-time setting}\label{sup:closed-by}

Suppose the hypothesis family is fixed, $I_t=[m]$, and $P_i=(P_{i,t})_{t\geq0}$ is a nonincreasing global $p$-process that is valid under optional stopping for $H_i$. For $k\geq1$, let $h_k=\sum_{j=1}^k j^{-1}$ and define
\[
  q_{k,\alpha}(p)
  :=\frac{k\,\one\{h_kp\leq\alpha\}}
  {\alpha\bigl(\lceil kh_kp/\alpha\rceil\vee1\bigr)}.
\]
For every nonempty $S\subseteq[m]$, put
\[
  E_t^{S,\mathrm{cBY}}
  :=\frac1{|S|}\sum_{i\in S}q_{|S|,\alpha}(P_{i,t}).
\]
For fixed $k$, the map $q_{k,\alpha}$ is a p-to-e calibrator used in the closed-BY construction of \citet{XuEtAl2026Closure}, which starts from the arbitrary-dependence BY correction of \citet{BenjaminiYekutieli2001}. Hence, under $P\in H_S$ and at every global stopping time $\tau$, each stopped term has expectation at most one and their average is an $e$-value. Because $P_i$ is nonincreasing and $q_{k,\alpha}$ is nonincreasing, the intersection processes are time-monotone. The persistent dynamic $e$-closure principle therefore gives simultaneous SupFDR control and setwise persistence. At every fixed time, the candidate family contains the ordinary BY rejection set by the fixed-time closed-BY result.

This is a batch any-time construction. The intersection-size-dependent arithmetic average changes under neutral padding, so the family size must be fixed.

\subsection{Persistent certificates and betting wealth}\label{sup:betting}

An $e$-process need not itself be a test supermartingale. Both provide any-time-valid evidence, but a test supermartingale has the stronger one-step property
\[
\E[M_t\mid\F_{t-1}]\leq M_{t-1},
\]
which supports a direct interpretation of $M_t$ as betting wealth under the null. The persistent certificates used in this paper need not have this property.

\begin{fact}[A nondecreasing supermartingale is constant]\label{sup:prop:constant}
	Let $(M_t)$ be an integrable supermartingale. If $M_t\geq M_{t-1}$ almost surely for every $t$, then $M_t=M_{t-1}$ almost surely for every $t$.
\end{fact}

\begin{proof}
	The supermartingale property gives
	\[
	\E[M_t-M_{t-1}\mid\F_{t-1}]\leq0.
	\]
	The increment is nonnegative, so its conditional expectation is also nonnegative and must equal zero. Hence the increment itself is zero almost surely.
\end{proof}

Thus a nonconstant time-monotone $e$-process cannot itself be a test supermartingale. While this does not prevent it from serving as any-time-valid evidence; it makes it weird to interpret it as the wealth process of a sequential betting strategy.

\subsection{Time-monotone \texorpdfstring{$e$}{e}-processes need not be adjusted running maxima}\label{sup:nonadjuster}

Running-maximum calibration originates in the evidence-insurance literature \citep{DawidEtAl2011Evidence,ShaferEtAl2011TestMartingales} and is used for persistent multiple testing by \citet{TavyrikovGoemanDeHeide2026}. Proposition~\ref{prop:monotone-eprocess-characterization} characterizes the larger class directly: a nonnegative adapted pathwise nondecreasing process is an $e$-process exactly when its terminal supremum has expectation at most one under the null. If one additionally imposed $E_0=1$, monotonicity would give $E_t\geq1$ almost surely and deterministic-time validity would force $E_t=1$ almost surely. Nontrivial time-monotone certificates therefore require the subnormalized convention used here.

\begin{proposition}[Adjusted running maxima are not exhaustive]\label{sup:prop:nonadjuster}
There exist nonconstant pathwise nondecreasing $e$-processes that cannot be written in the form
\[
  V_t=A\!\left(1\vee\sup_{0\leq u\leq t}E_u\right)
\]
for any adjuster $A$ and any normalized $e$-process $E$ with $E_0=1$.
\end{proposition}

\begin{proof}
Let $V_0=0$ and $V_t=1/2$ for every $t\geq1$. This deterministic process is adapted, pathwise nondecreasing, and satisfies $\E[V_\tau]\leq1/2\leq1$ for every stopping time $\tau$, so it is an $e$-process.

Suppose that the displayed representation held. At time zero, normalization of $E$ gives
\[
  0=V_0=A(1).
\]
At time one, $V_1=1/2$ almost surely. Since $A$ is nondecreasing and $A(1)=0$, this forces
\[
  1\vee\max(E_0,E_1)>1
  \qquad\text{almost surely}.
\]
Because $E_0=1$, we obtain $E_1>1$ almost surely, and hence $\E[E_1]>1$. This contradicts the $e$-process property at the deterministic stopping time one.
\end{proof}

Hence adjusted running maxima do not exhaust the class of time-monotone $e$-processes.

\end{document}